\documentclass{amsart}

\usepackage{amssymb}
\usepackage{amsthm}
\usepackage{color}
\usepackage{algorithm}
\usepackage{algorithmic}
\usepackage{graphicx}
\usepackage{subfigure}
\usepackage{amsmath}
\usepackage{multirow}
\usepackage{epstopdf}
\usepackage{tikz}
\input{undertilde}
\usetikzlibrary{snakes}

\usepackage{verbatim}
\usepackage{ulem}

\newcommand{\bs}{\boldsymbol}

\def\uphi{\undertilde{\varphi}}
\def\upsi{\undertilde{\psi}}

\def\uH{\undertilde{H}}

\def\rot{{\rm rot}}

\def\dv{{\rm div}}

\def\vgm12{\bs{V}^{1+,2}_{\gamma,M}}

\newtheorem{theorem}{Theorem}
\newtheorem{remark}[theorem]{Remark}

\newtheorem{lemma}[theorem]{Lemma}

\begin{document}

\title[A multi-level scheme of transmission eigenvalue problem]{A multi-level mixed element scheme of the two dimensional Helmholtz transmission eigenvalue problem}

\author{Yingxia Xi}
\address{LSEC, Institute of Computational Mathematics and Scientific/Engineering Computing, Academy of Mathematics and System Sciences, Chinese Academy of Sciences, Beijing 100190, People's Republic of China}
\email{yxiaxi@lsec.cc.ac.cn}

\author{Xia Ji}
\address{LSEC, Institute of Computational Mathematics and Scientific/Engineering Computing, Academy of Mathematics and System Sciences, Chinese Academy of Sciences, Beijing 100190, People's Republic of China}
\email{jixia@lsec.cc.ac.cn}
\thanks{X. Ji is supported by the National Natural Science Foundation of China (No. 11271018, No. 91630313) and National Centre for Mathematics and Interdisciplinary Sciences, Chinese Academy of Sciences.}

\author{Shuo Zhang}
\address{LSEC, Institute of Computational Mathematics and Scientific/Engineering Computing, Academy of Mathematics and System Sciences, Chinese Academy of Sciences, Beijing 100190, People's Republic of China}
\email{szhang@lsec.cc.ac.cn}
\thanks{S. Zhang is partially supported by the National Natural Science Foundation of China with Grant No. 11471026 and National Centre for Mathematics and Interdisciplinary Sciences, Chinese Academy of Sciences.}



\begin{abstract}
{In this paper, we present a multi-level mixed element scheme for the Helmholtz transmission eigenvalue problem on polygonal domains that are not necessarily able to be covered by rectangle grids. We first construct an equivalent linear mixed formulation of the transmission eigenvalue problem and then discretize it with Lagrangian finite elements of low regularities. The proposed scheme admits a natural nested discretization, based on which we construct a multi-level scheme. Optimal convergence rate and optimal computational cost can be obtained with the scheme.}
\end{abstract}

\maketitle

\section{Introduction}

In this paper, we study the numerical method of the Helmholtz transmission eigenvalue problem in two dimension. For the scattering of time-harmonic acoustic waves by a bounded simply connected inhomogeneous medium $\Omega \subset \mathcal{R}^2$, the transmission eigenvalue
problem is to find $k\in\mathcal{C}$, $\phi, \varphi\in H^2(\Omega)$ such that
\begin{equation}\label{Problem1}
\left\{
\begin{array}{rcll}
\Delta \phi+k^2n(x)\phi&=&0,&{\rm in}\ \Omega,\\
\Delta \varphi+k^2\varphi&=&0, &{\rm in}\ \Omega,\\
\phi-\varphi&=&0,&{\rm on}\ \partial \Omega,\\
\frac{\partial \phi}{\partial \nu}-\frac{\partial \varphi}{\partial\nu}&=&0,&{\rm on}\ \partial \Omega,
\end{array}
\right.
\end{equation}
where $\nu$ is the unit outward normal to the boundary $\partial \Omega$. The index of refraction $n(x)$ is assumed to be positive. 
Values of $k$ such that there exists a nontrivial solution to (\ref{Problem1}) are called transmission eigenvalues.

The transmission eigenvalue problem is arising in inverse scattering theory \cite{CakoniColtonMonkSun,CakoniHaddar,ColtonKress,ColtonMonkSun,ColtonPaivarintaSykvester,Sun2016}. Since the transmission eigenvalues can be determined from the far field pattern \cite{Colton1988}, they can be used to obtain estimates for the material properties of the scattering object \cite{CakoniCayorenColton}. Furthermore, transmission eigenvalues have theoretical importance in the uniqueness and reconstruction in inverse scattering theory \cite{ColtonKress}. The problem thus has been attracting wide interests on the mathematical and numerical analysis.

Contrast to some existing model problems, the Helmholtz transmission eigenvalue problem is non-self-adjoint, and thus all classical theoretical tools can not be applied directly. It can be proved that the transmission eigenvalues form at most a discrete set with infinity as the only possible accumulation point by applying the analytic Fredholm theory \cite{ColtonPaivarintaSykvester}. However, little was known about the existence of the transmission eigenvalues except some special cases. In \cite{PaivarintaSykvester}, P\"{a}iv\"{a}rinta and Sylvester show the existence of a finite number of transmission eigenvalues provided that the index of refraction is large enough. Cakoni and Haddar \cite{CakoniHaddar2009} extend the idea of \cite{PaivarintaSykvester} and prove the existence of finitely many transmission eigenvalues for a larger class of problems. The idea is further extended to show the existence of an infinite discrete set of transmission eigenvalues that accumulate at infinity \cite{Cakoni2012}.

Besides, there is an infinite-dimensional eigenspace corresponding to the nonphysical transmission eigenvalue k = 0. Actually, it is readily seen that
any harmonic function on $\Omega$ is an eigenfunction by setting $k = 0$ in  (\ref{Problem1}) such that the first equation and the second equation become the same. Following \cite{JiSunTurner}, we define
$
V:=H_0^2(\Omega)=\Big\{u\in H^2(\Omega):\ u=0\ {\rm and}\ \frac{\partial u}{\partial{\nu}}=0\ {\rm on}\
 \partial\Omega\Big\}
$  and introduce a new variable $u=\phi-\varphi\in V$, and then $u$ and $k$ satisfy the fourth order problem
\begin{eqnarray}\label{Fourth_Order_Eigenvalue_Problem1}
\big(\Delta +k^2n(x)\big)\frac{1}{n(x)-1}(\Delta +k^2)u=0.
\end{eqnarray}
It is obvious that $k = 0$ is not a nontrivial eigenvalue of the eigenvalue problem \eqref{Fourth_Order_Eigenvalue_Problem1} any longer. The nonlinear eigenvalue problem \eqref{Fourth_Order_Eigenvalue_Problem1} is then a physically consistent formulation. 

As \eqref{Fourth_Order_Eigenvalue_Problem1} falls into the category of fourth order problems, the conforming Argyris element method, proposed by \cite{ColtonMonkSun}, is a natural approach. The BFS element was analogously discussed for rectangular grids in \cite{JiSunXie,YangHanBi2}. Due to the high compliancy of the conforming elements, the nonconforming Morley element method was studied in \cite{JiXiXie}. Further, the discontinuous Galerkin (DG) method, such as the $C^0$-IPG method using standard $C^0$ Lagrange finite elements was also applied on the transmission eigenvalue problem \cite{GengJiSunXu}. A mixed element method was discussed in \cite{JiSunTurner}. For this method, only $C^0$ finite elements are required. Some methods other than finite element methods were also reported, such as the recursive integral method (RIM) proposed in
\cite{HuangStruthersSunZhang,XiJi}. The related source problem \cite{HsialLiuSunXu, WuChen} and other multi-level type methods \cite{JiSun2013JCP} have also been discussed.

In this paper, we present a multi-level mixed element method of \eqref{Fourth_Order_Eigenvalue_Problem1}. As well known, the multi-level algorithm based on nested essence has been a key tool in the fields of computational mathematics and scientific computing. For eigenvalue problems, many multi-level algorithms have been designed and implemented. A type of multi-level scheme is presented by Lin-Xie \cite{LinXie,Xie}. The method is related to \cite{Lin.Q1979,XuZhou2001,xu1992new,xu1994novel}, and has presented a framework of designing multi-level schemes which works well for the elliptic eigenvalue problem and stable saddle point problem, provided a series of subproblems with intrinsic nestedness constructed. For the fourth order problem in primal formulations where the second order Sobolev spaces are involved, the discretizations can hardly be nested. The only known nested finite element other than spline type ones is the BFS element on rectangular grids, and its multi-level algorithm has been discussed for \eqref{Fourth_Order_Eigenvalue_Problem1} by \cite{JiSunXie,YangHanBi2}, but no results are known on triangular grids. In this paper, we will first construct a mixed element method for \eqref{Fourth_Order_Eigenvalue_Problem1}. The newly constructed mixed formulation employs Sobolev spaces of zeroth and first orders only such that nested hierarchy can be naturally expected, and then we implement Lin-Xie's framework onto the formulation to construct a multi-level algorithm on triangular grids. Optimal accuracy and optimal computational cost can be obtained.

We remark that for the mixed element method for \eqref{Fourth_Order_Eigenvalue_Problem1} presented in \cite{JiSunTurner}, it remains open whether and how this method is equivalent to the primal formulation \eqref{Fourth_Order_Eigenvalue_Problem1}, especially on non-convex domains. Moreover, the generated schemes is not topologically nested. In our present method, the order reduced formulation is equivalent to the primal formulation, and the generated schemes are topologically nested, and thus Lin-Xie's framework can be utilised.

The remaining of the paper is organized as follows. In Section 2, we collect some necessary preliminaries. In Section 3, we present a mixed formulation of the transmission eigenvalue problem. The equivalence between the linear order reduced formulation and the primal formulation is proved. Section 4 constructs the discretization scheme and a multi-level algorithm follows. Numerical examples are then given in Section 5 with the discussion about complex eigenvalues. Finally, in Section 6 some concluding remarks are given.

\section{Preliminaries}
\subsection{Transmission eigenvalue problem}

For the scattering of time-harmonic acoustic waves by a bounded simply connected inhomogeneous medium $\Omega \subset \mathcal{R}^2$, the transmission eigenvalue
problem is to find $k\in\mathcal{C}$, $\phi, \varphi\in H^2(\Omega)$ such that
\begin{equation*}
\left\{
\begin{array}{rcll}
\Delta \phi+k^2n(x)\phi&=&0,&{\rm in}\ \Omega,\\
\Delta \varphi+k^2\varphi&=&0, &{\rm in}\ \Omega,\\
\phi-\varphi&=&0,&{\rm on}\ \partial \Omega,\\
\frac{\partial \phi}{\partial \nu}-\frac{\partial \varphi}{\partial\nu}&=&0,&{\rm on}\ \partial \Omega,
\end{array}
\right.
\end{equation*}
where $\nu$ is the unit outward normal to the boundary $\partial \Omega$.   Following the same procedure in
\cite{JiSunTurner}, introducing a new variable $u=\phi-\varphi\in V$, $u$ satisfies
\begin{equation*}
(\triangle + k^2)u = k^2(1-n(x))\phi,\quad\mbox{namely}\quad \frac{1}{1-n(x)}(\triangle + k^2)u=k^2 \phi.
\end{equation*}
We apply $(\triangle +k^2n(x))$ to both sides of the above equation to obtain
\begin{equation*}
\big(\Delta +k^2n(x)\big)\frac{1}{n(x)-1}(\Delta +k^2)u=0.
\end{equation*}
Note that $k = 0$ is not a nontrivial eigenvalue any longer, since $\displaystyle (\frac{1}{n(x)-1}\Delta u, \Delta u)=0$ and $u \in V$ implies that $u = 0$. On the other hand, it's easy to see the equivalence by setting $\phi=\displaystyle\frac{1}{k^2(1-n(x))}(\triangle + k^2)u$ and $\varphi=\phi-u$.

The variational formulation of the transmission eigenvalue problem is to find $(k^2\neq 0,u)\in\mathbb{C}\times V$, such that
\begin{equation}\label{eq:vftep}
\left(\frac{1}{n(x)-1}(\Delta u+k^2u),\Delta v+k^2n(x)v\right)=0,\ \ \forall\, v\in V.
\end{equation}
Here $0<n_s\leqslant n(x)\leqslant n_b$.

\subsection{Fundamental results of spectral approximation of compact operators}

In this subsection, we present some fundamental results on the spectral approximation of compact operators. They can be found from \cite{Babuska.I;Osborn.J1991}.

First of all, we introduce the symbol $\eqslantless$ to denote an order of complex numbers.  Let $\mathsf{c}_k=\rho_ke^{i\theta_k}$, $k=1,2$, be two complex numbers, with $\rho_k\geqslant 0$ and $0\leqslant \theta_k<2\pi$. Then $\mathsf{c}_1\eqslantless\mathsf{c}_2$ if and only if one of the items below holds:
\begin{enumerate}
\item $\rho_1=\rho_2=0$;
\item $\rho_1<\rho_2$;
\item $\rho_1=\rho_2\neq0$, and $\theta_1\geqslant \theta_2$.
\end{enumerate}
It is evident that if $\mathsf{c}_1\eqslantless\mathsf{c}_2$ and $\mathsf{c}_2\eqslantless\mathsf{c}_3$, then $\mathsf{c}_1\eqslantless\mathsf{c}_3$.

Coherently, we use the symbol $``\eqslantgtr"$, whereas $\mathsf{c}_2\eqslantgtr\mathsf{c}_1$ if and only if $\mathsf{c}_1\eqslantless\mathsf{c}_2$.

\begin{lemma}
(\cite{Chang.K;Lin.Y1990})
Let $T$ be a compact operator on the Banach space $X$, then all its eigenvalues, counting multiplicity, can be listed in a (finite or infinite) sequence as
\begin{equation}\label{eq:ordered}
\mu_1\eqslantgtr\mu_2\eqslantgtr\dots\eqslantgtr0.
\end{equation}
\end{lemma}

\begin{lemma}
Let $\{T_h\}$ be a family of compact operators on $X$, such that $\|T_h-T\|_{X\to X}\to 0$ as $h$ tends to zero. List the eigenvalues of $T_h$ in a sequence as
\begin{equation}
\mu_1^h\eqslantgtr\mu_2^h\eqslantgtr\dots\eqslantgtr0.
\end{equation}
Then $\lim_{h\to 0}\mu_i^h=\mu_i$ for any $i$.
\end{lemma}

A {\bf gap} between two closed subspaces $M$ and $N$ of $X$ is defined by
$$
\hat{\delta}(M,N)=\max(\delta(M,N),\delta(N,M)),  \mbox{with}\ \delta(M,N)=\sup_{x\in M, \|x\|=1}{\rm dist}(x,N).
$$

\begin{lemma}
Let $\mu_i$ be a nonzero eigenvalue of $T$. Then
\begin{equation}
\hat\delta(M(\mu_i),M_h(\mu_i^h))\leqslant C\|(T-T_h)|_{M(\mu_i)}\|_{X\to X},
\end{equation}
\end{lemma}
where $M(\mu_i)$ and $M_h(\mu_i^h)$ respectively are the eigenspace corresponding to $\mu_i$ and $\mu_i^h$.
\begin{remark}
Particularly, we consider there are a family of subspaces $\{X_h\}_{h>0}$ of $X$, and a family of idempotent operators $\{P_h\}_{h>0}$ from $X$ onto $X_h$, such that $T_h=P_hT$. Then
\begin{equation}
\hat\delta(M(\mu_i),M_h(\mu_i^h))\leqslant C\|(Id-P_h)|_{M(\mu_i)}\|_{X\to X}.
\end{equation}
\end{remark}

\subsubsection{A multi-level scheme for the eigenvalue problem}\label{sec:abstractml}
{\bf Algorithm 1} presents a multi-level scheme for computing the first $k$ (as ordered in (\ref{eq:ordered})) eigenvalues of a compact operator $T$. The algorithm is the scheme by Lin-Xie \cite{LinXie,Xie} rewritten in the operator formulation.

\begin{algorithm}[htbp]
\label{alg:mlalg}
\caption{A multi-level algorithm for the first $k$ eigenvalues of $T$.}
%
\begin{description}
\item[Step 0] Construct a series of nested spaces $G_0\subset G_1\subset\dots\subset G_N\subset X$. Set $\widetilde{G}_0=G_0$.
\item[Step 1] For $i=1:1: N$, generate auxiliary spaces $\widetilde{G}_i$ recursively.
\begin{description}
\item[Step 1.i.1]
Define idempotent operators $\widetilde{P}_{i-1}:H\to \widetilde{G}_{i-1}$, and solve eigenvalue problem for its first $k$ eigenpairs $\{(\tilde{\mu}_j^{i-1},\tilde{u}_j^{i-1})\}_{j=1,\dots,k}$
$$
 \widetilde{P}_{i-1}T\tilde{u}=\tilde{\mu}\tilde{ u};
$$
\item[Step 1.i.2] Define idempotent operators $P_i:H\to G_i$. Compute
$$
\hat{u}_j^i=\frac{1}{\tilde{\mu}_j^{i-1}}P_iT\tilde{u}_j^{i-1},\ \ j=1,\dots,k;
$$
\item[Step 1.i.3] Set
$$
\widetilde{G}_i=G_0+{\rm span}\{\hat{u}_j^i\}_{j=1}^k.
$$
\end{description}
\item[Step 2] Define idempotent operators $\widetilde{P}_N:H\to \widetilde{G}_N$, solve eigenvalue problem for its first $k$ eigenpairs $\{(\tilde{\mu}_j^N,\tilde{u}_j^N)\}_{j=1,\dots,k}$:
$$
\widetilde{P}_NT\tilde{u}=\tilde{\mu}\tilde{ u}.
$$
\end{description}
\end{algorithm}
\section{Mixed formulation of the transmission eigenvalue problem}

In this section, we present a stable equivalent mixed formulation of the transmission eigenvalue problem. We discuss the cases $n_s>1$ on $\Omega$ and $n_b<1$ on $\Omega$ separately. As the original problem is a quadratic eigenvalue problem on $H^2$ space, we will adopt a two-step process to transform the problem to a linear order reduced formulation.

%
%
\subsection{Case I: $n_s>1$}

\subsubsection{Step I: on linearization of the eigenvalue problem}

If $n(x)>1$, we rewrite the problem as
\begin{equation}\label{eq:vftepng1}
\left(\frac{1}{n(x)-1}(\Delta u+k^2u),(\Delta v+k^2v)\right)+k^4(u,v)-k^2(\nabla u,\nabla v)=0,\ \ \forall\, v\in H^2_0(\Omega).
\end{equation}
Writing $\lambda=k^2$, $y=\lambda u$, $z=\lambda v$ and $\alpha:=\frac{1}{n(x)-1}$, we are going to find $(\lambda,u)\in\mathbb{C}\times H^2_0(\Omega)$, such that
\begin{equation}
\left(\alpha(\Delta u+y),(\Delta v+z)\right)+(y,z)-\lambda (\nabla u,\nabla v)=0,\ \mbox{and}\ y=\lambda u,\ \ \forall\, v\in H^2_0(\Omega),\ z=\lambda v.
\end{equation}
The variational problem is to find $(u,y,p)\in U:=H^2_0(\Omega)\times L^2(\Omega)\times L^2(\Omega)$, such that, for $(v,z,q)\in U$,
\begin{equation}\label{eq:eigponce}
\left\{
\begin{array}{cccccccl}
(\alpha\Delta u,\Delta v)&+(\alpha y,\Delta v)&&=&\lambda(\nabla u,\nabla v)&&-\lambda(p,v)&
\\
(\alpha\Delta u,z)&+((1+\alpha)y,z)&-(p,z) &=& 0&
\\
&-(y,q)& &=& -\lambda(u,q).&&&
\end{array}
\right.
\end{equation}
Define
\begin{equation}
a_{U,\alpha}((u,y,p),(v,z,q)):=(\alpha\Delta u,\Delta v)+(\alpha y,\Delta v)+(\alpha\Delta u,z)+((1+\alpha)y,z)-(p,z) -(y,q).
\end{equation}

\begin{lemma}
$a_{U,\alpha}(\cdot,\cdot)$ is continuous on $U$, and
\begin{equation}
\inf_{(v,z,q)\in U}\sup_{(u,y,p)\in U}\frac{a_{U,\alpha}((u,y,p),(v,z,q))}{\|(u,y,p)\|_U\|(v,z,q)\|_U}\geqslant C>0.
\end{equation}
\end{lemma}
\begin{proof}
By elementary calculation, $(\alpha\Delta u, \Delta u)+2(\alpha\Delta u,y)+((1+\alpha)y,z) \geqslant \frac{1}{n_b-1} \left(1-\sqrt{1\over n_s}\right) (\|y\|_{0,\Omega}^2+\|\Delta u\|_{0,\Omega}^2)$ for $u\in H^2_0(\Omega)$ and $y\in L^2(\Omega)$.
It is evident that $\displaystyle\inf_{q\in L^2}\sup_{y\in L^2}\frac{(y,q)}{\|y\|_{0,\Omega}\|q\|_{0,\Omega}}=1.$ The proof is completed by Babuska-Brezzi theory.
\end{proof}

Define
$$
b_U((u,y,p),(v,z,q)):=(\nabla u,\nabla v)-(p,v)-(q,u).
$$
Then $b_U(\cdot,\cdot)$ is symmetric and continuous on $U$.

Define $T:V\to V$ by
\begin{equation}
a_{U,\alpha}(T(u,y,p),(v,z,q)):=b_U((u,y,p),(v,z,q)).
\end{equation}
\begin{lemma}
$T$ is well-defined, and $T$ is compact.
\end{lemma}
\begin{proof} Evidently, $\|T(u,y,p)\|_V\leqslant C(\|-\Delta u-p\|_{-2,\Omega}+\|u\|_{0,\Omega})\leqslant C(\|u\|_0+\|p\|_{-2,\Omega})$.
Now, let $\{(u_j,y_j,p_j)\}$ be a bounded sequence in $V$, then there is subsequence $\{(u_{j_k},y_{j_k},p_{j_k})\}$, such that $\{u_{j_k}\}$ is a Cauchy sequence in $L^2(\Omega)$, and $\{p_{j_k}\}$ is a Cauchy sequence in $H^{-2}(\Omega)$. Therefore,  $\{T(u_{j_k},y_{j_k},p_{j_k})\}$ is a Cauchy sequence in $V$, which, further, has a limit therein. This finishes the proof.
\end{proof}

\subsubsection{On the order reduction of $H^2$}

By writing $\uphi:=\nabla u$ and introducing Lagrangian multipliers $\sigma$ and $r$, we rewrite the eigenvalue problem \eqref{eq:eigponce} as: find $(y,\uphi,u,p,\sigma,r)\in V:=L^2(\Omega)\times \undertilde{H}{}^1_0(\Omega)\times H^1_0(\Omega)\times L^2(\Omega)\times L^2_0(\Omega)\times H^1_0(\Omega)$ and $\lambda\in\mathbb{C}$, such that, for any $(z,\upsi,v,q,\tau,s)\in V$,
\begin{equation}\label{eq:relaxbig}
a_\alpha((y,\uphi,u,p,\sigma,r),(z,\upsi,v,q,\tau,s))=\lambda b((y,\uphi,u,p,\sigma,r),(z,\upsi,v,q,\tau,s)),
\end{equation}
where
\begin{multline}
a_\alpha((y,\uphi,u,p,\sigma,r),(z,\upsi,v,q,\tau,s))
\\
:=((1+\alpha)y,z)+(\alpha\dv\uphi,z)-(p,z)+(\alpha y,\dv\upsi)
+ (\alpha\dv\uphi,\dv\upsi)+(\rot\uphi,\rot\upsi)
\\
+ (\sigma,\rot\upsi)-(\nabla r,\upsi)+(\nabla r,\nabla v)-(y,q)
+(\rot\uphi,\tau)-(\uphi,\nabla s)+(\nabla u,\nabla s),
\end{multline}
and
\begin{equation}
b((y,\uphi,u,p,\sigma,r),(z,\upsi,v,q,\tau,s)):=(\uphi,\nabla v)-(p,v)-(u,q).
\end{equation}
Define $T_\alpha:V\to V$ by
\begin{equation}
a_\alpha(T_\alpha(y,\uphi,u,p,\sigma,r),(z,\upsi,v,q,\tau,s))=b((y,\uphi,u,p,\sigma,r),(z,\upsi,v,q,\tau,s)),\ \ \forall\,(z,\upsi,v,q,\tau,s)\in V.
\end{equation}
\begin{lemma}\label{TCompact}
$T_\alpha$ is a compact operator from $V$ to $V$.
\end{lemma}

\begin{proof}
It is evident that $a_\alpha(\cdot,\cdot)$ and $b(\cdot,\cdot)$ are bounded on $V$. Now we rewrite \eqref{eq:relaxbig} to an expanded formulation:
\begin{small}
\begin{equation}\label{eq:stmexpd}
\left\{
\begin{array}{cccccclll}
((1+\alpha)y,z)&+(\alpha\dv\uphi,z)&&-(p,z)&&&=&0
\\
(\alpha y,\dv\upsi)& +(\alpha\dv\uphi,\dv\upsi)+(\rot\uphi,\rot\upsi) & & & +(\sigma,\rot\upsi)&-(\nabla r,\upsi)&=&0
\\
&&&&&(\nabla r,\nabla v)&=&\lambda(\uphi,\nabla v)-\lambda(p,v)
\\
-(y,q)&&&&&&=&-\lambda(u,q)
\\
&(\rot\uphi,\tau)&&&&&=& 0
\\
&-(\uphi,\nabla s)&+(\nabla u,\nabla s)&&&&=&0.
\end{array}
\right.
\end{equation}
\end{small}
Denote
$A((y,\uphi,u),(z,\upsi,v)):=((1+\alpha)y,z)+(\alpha\dv\uphi,z)+(\alpha y,\dv\upsi)+(\alpha\dv\uphi,\dv\upsi)+(\rot\uphi,\rot\upsi),$ then, again, with elementary calculation,
$$
A((y,\uphi,u),(y,\uphi,u))\geqslant \frac{1}{n_b-1}\left(1-\sqrt{1\over n_s}\right)(\|y\|_{0,\Omega}^2+\|\dv\uphi\|_{0,\Omega}^2)+\|\rot\uphi\|_{0,\Omega}^2.
$$
Further denote $B((y,\uphi,u),(q,\tau,s)):=-(y,q)+(\rot\uphi,\tau)-(\uphi,\nabla s)+(\nabla u,\nabla s)$, and $Z:=\{(y,\uphi,u)\in L^2(\Omega)\times \undertilde{H}{}^1_0(\Omega)\times H^1_0(\Omega):B((y,\uphi,u),(q,\tau,s))=0,\ \forall\,(q,\tau,s)\in L^2(\Omega)\times L^2_0(\Omega)\times H^1_0(\Omega)\}$. Then
$$
A((y,\uphi,u),(y,\uphi,u))\geqslant C(\|y\|_{0,\Omega}^2+\|\uphi\|_{1,\Omega}^2+\|u\|_{1,\Omega}^2)\ \mbox{on}\ Z.
$$

Meanwhile, given $(q,\tau,s)\in L^2(\Omega)\times L^2_0(\Omega)\times H^1_0(\Omega)$, take $y=-q$, $\uphi\in\uH{}^1_0(\Omega)$ such that $\rot\uphi=\tau$ and $\|\uphi\|_{1,\Omega}\leqslant C\|\rot\uphi\|_{0,\Omega}$, and $u\in H^1_0(\Omega)$, such that $(\nabla u,\nabla v)-(\uphi,\nabla v)=(\nabla s,\nabla v)$ for any $v\in H^1_0(\Omega)$. Then $B((y,\uphi,u),(q,\tau,s))=(q,q)+(\tau,\tau)+(\nabla s,\nabla s)$ and $\|y\|_{0,\Omega}+\|\uphi\|_{1,\Omega}+\|u\|_{1,\Omega}\leqslant C(\|q\|_{0,\Omega}+\|\tau\|_{0,\Omega}+\|s\|_{1,\Omega})$. This proves the coercivity
\begin{small}
\begin{equation}
\sup_{(y,\uphi,u)\in L^2(\Omega)\times \undertilde{H}{}^1_0(\Omega)\times H^1_0(\Omega)\setminus\{\bf 0\}}\frac{B((y,\uphi,u),(q,\tau,s))}{(\|y\|_{0,\Omega}+\|\uphi\|_{1,\Omega}+\|u\|_{1,\Omega})(\|q\|_{0,\Omega}+\|\tau\|_{0,\Omega}+\|s\|_{1,\Omega})}\geqslant C,
\end{equation}
\end{small}
for any $(q,\tau,s)\in L^2(\Omega)\times L^2_0(\Omega)\times H^1_0(\Omega)\setminus\{\mathbf{0}\}$.
This confirms the well-posed-ness of $T_\alpha$.

Now, let $\{(y_i,\uphi{}_i,u_i,p_i,\sigma_i,r_i)\}$ be a bounded sequence in $V$, then there is subsequence, labelled as $\{(y_{i_k},\uphi{}_{i_k},u_{i_k},p_{i_k},\sigma_{i_k},r_{i_k})\}$, such that $\{\uphi{}_{i_k}\}$, $\{u_{i_k}\}$ and $\{p_{i_k}\}$ are three Cauchy sequences in $L^2(\Omega)$, $L^2(\Omega)$ and $H^{-1}(\Omega)$ respectively. Therefore,  $\{T_\alpha(y_{i_k},\uphi{}_{i_k},u_{i_k},p_{i_k},\sigma_{i_k},r_{i_k})\}$ is a Cauchy sequence in $V$, which, further, has a limit therein. This finishes the proof.
\end{proof}

\begin{theorem}\label{thm:evpequiv}
The eigenvalue problem \eqref{eq:relaxbig} is equivalent to the eigenvalue problem \eqref{eq:vftepng1}.
\end{theorem}
\begin{proof}
If $\lambda$ and $(y,\uphi,u,p,\sigma,r)$ is a solution of \eqref{eq:relaxbig}, then $y=\lambda u$, $\uphi=\nabla u$ and $u\in V$, $\lambda$ and $u$ solves \eqref{eq:vftepng1}. Meanwhile, if $\lambda$ and $u$ solves \eqref{eq:vftepng1}, then substituting $y=\lambda u$ and $\uphi=\nabla u$ into the system \eqref{eq:stmexpd}, a unique $(p,\sigma,r)\in L^2(\Omega)\times L^2_0(\Omega)\times H^1_0(\Omega)$ can be determined. The equivalence is confirmed, and the proof is completed. \end{proof}

\begin{remark}
The choice of $b(\cdot,\cdot)$, of course, is not unique. For example, define
\begin{equation}
\hat b((y,\uphi,u,p,\sigma,r),(z,\upsi,v,q,\tau,s)):=(\nabla u,\nabla v)-(p,v)-(u,q),
\end{equation}
then the equation
\begin{equation}\label{eq:relaxbig'}
a_\alpha((y,\uphi,u,p,\sigma,r),(z,\upsi,v,q,\tau,s))=\lambda \hat b((y,\uphi,u,p,\sigma,r),(z,\upsi,v,q,\tau,s)),
\end{equation}
has the same solution as \eqref{eq:relaxbig}. A difference can lie in utilizing the compact operator argument with \eqref{eq:relaxbig'}.
\end{remark}

\subsection{Case II: $n_b<1$}

For the case $n_b<1$, the procedure is the same as that for the case $n_s>1$. The only difference is that we rewrite the original problem to:
\begin{equation}
\left(\frac{n(x)}{1-n(x)}(\Delta u+k^2u),(\Delta v+k^2v)\right)+(\Delta u,\Delta v)-k^2(\nabla u,\nabla v)=0,\ \ \forall\, v\in H^2_0(\Omega).
\end{equation}
Below, we only list the main results, and omit the proof. Set $\displaystyle \beta=\frac{n(x)}{1-n(x)}$, then $\displaystyle \frac{n_s}{1-n_s}\leqslant \beta(x)\leqslant \frac{n_b}{1-n_b}$. Define
\begin{multline}
a_\beta((y,\uphi,u,p,\sigma,r),(z,\upsi,v,q,\tau,s))
\\
:=(\beta y,z)+(\beta\dv\uphi,z)-(p,z)+(\beta y,\dv\upsi)+ ((1+\beta)\dv\uphi,\dv\upsi)+(\rot\uphi,\rot\upsi)
\\
+ (\sigma,\rot\upsi)-(\nabla r,\upsi)+(\nabla r,\nabla v)-(y,q)+(\rot\uphi,\tau)-(\uphi,\nabla s)+(\nabla u,\nabla s),
\end{multline}
and
$T_\beta:V\to V$ by
$$
a_\beta(T_\beta(y,\uphi,u,p,\sigma,r),(z,\upsi,v,q,\tau,s))=b((y,\uphi,u,p,\sigma,r),(z,\upsi,v,q,\tau,s)),\ \ \forall\,(z,\upsi,v,q,\tau,s)\in V.
$$
$T_\beta$ is also a compact operator from $V$ to $V$. The following theorem gives a consistent one-to-one match of eigenvalues between the primal eigenvalue system and the compact operator $T_\beta$.
\begin{theorem}
If $n_b<1$, the primal transmission eigenvalue problem (\ref{Problem1}) is equivalent to find $(y,\uphi,u,p,\sigma,r)\in V$ and $\lambda\in\mathbb{C}$, such that
$$
a_\beta((y,\uphi,u,p,\sigma,r),(z,\upsi,v,q,\tau,s))=\lambda b((y,\uphi,u,p,\sigma,r),(z,\upsi,v,q,\tau,s)), \quad \forall\,(z,\upsi,v,q,\tau,s)\in V.
$$
\end{theorem}
\section{Discretization}
We discuss the case $n_s>1$ for illustration, and the case $n_b<1$ is the same.
\subsection{Discretisation schemes of \eqref{eq:relaxbig}}
To discretize \eqref{eq:relaxbig}, we have to discretize $L^2$ (twice), $\undertilde{H}{}^1_0(\Omega)$, $H^1_0(\Omega)$ (twice) and $L^2_0(\Omega)$. Let $L^2_{h}(\Omega)\subset L^2(\Omega)$, $H^1_{h0}\subset H^1_0(\Omega)$, $\undertilde{H}{}^{1}_{h0}\subset \undertilde{H}{}^{1}_0$, and $L^2_{h0}\subset L^2_0(\Omega)$ be respective finite element subspaces. Define
\begin{equation}
 V_{h}:=L^2_{h}(\Omega) \times \undertilde{H}{}^1_{h0} \times H^1_{h0}\times L^2_{h}(\Omega) \times L^2_{h0}\times H^1_{h0}.
\end{equation}
We introduce the discretized mixed eigenvalue problem:
find $\lambda_h\in \mathbb{C}$ and  $(y_h,\uphi{}_h,u_h,p_h,\sigma_h,r_h)\in V_{h}$,  such that, for $\forall (z_h,\upsi{}_h,v_h,q_h,\tau_h,s_h)\in V_{h}$,
\begin{equation}\label{eq:saddleevpdis}
a_\alpha((y_h,\uphi{}_h,u_h,p_h,\sigma_h,r_h),(z_h,\upsi{}_h,v_h,q_h,\tau_h,s_h)) = \lambda_h b((y_h,\uphi{}_h,u_h,p_h,\sigma_h,r_h),(z_h,\upsi{}_h,v_h,q_h,\tau_h,s_h)).
\end{equation}

For the well-posedness of the discretized problem, we propose the assumption below.

\paragraph{\bf Assumption AIS} The discrete inf-sup condition holds uniformly that
\begin{equation}\label{eq:assmdisis}
\inf_{q_h\in L^2_{h0}}\sup_{\upsi{}_h\in \undertilde{H}{}^{1}_{h0}}\frac{(\rot\upsi{}_h,q_h)}{\|\nabla_h\upsi{}_h\|_{0,\Omega}\|q_h\|_{0,\Omega}}\geqslant C.
\end{equation}
\begin{remark}
The condition \eqref{eq:assmdisis} is equivalent to the well studied inf-sup condition for the two-dimensional incompressible Stokes problem. It is sufficient to verify that for $\undertilde{H}{}^{1}_{h0}$.
\end{remark}

Associated with $a_\alpha(\cdot,\cdot)$ and $b(\cdot,\cdot)$, we define an operator $T_{\alpha, h}$ by
\begin{multline}\label{eq:defineTh_alpha}
\qquad a_\alpha(T_{\alpha,h}(y,\uphi{},u,p,\sigma,r),(z_h,\upsi{}_h,v_h,q_h,\tau_h,s_h))=b((y,\uphi{},u,p,\sigma,r),(z_h,\upsi{}_h,v_h,q_h,\tau_h,s_h)),\\
\forall\, (z_h,\upsi{}_h,v_h,q_h,\tau_h,s_h)\in V_h.\qquad
\end{multline}
and an operator $P_{\alpha,h}$ by
\begin{multline}
\qquad a_\alpha(P_{\alpha,h}(y,\uphi{},u,p,\sigma,r),(z_h,\upsi{}_h,v_h,q_h,\tau_h,s_h))=a_\alpha((y,\uphi{},u,p,\sigma,r),(z_h,\upsi{}_h,v_h,q_h,\tau_h,s_h)),\\
\forall\, (z_h,\upsi{}_h,v_h,q_h,\tau_h,s_h)\in V_h.\qquad
\end{multline}

Evidently, $T_{\alpha,h}=P_{\alpha,h}T_\alpha$. By the standard theory of finite element methods and by the same virtue of the proof of Lemma \ref{TCompact}, we have the lemma below.
\begin{lemma}
Provided the {\bf Assumption AIS} \eqref{eq:assmdisis},
\begin{enumerate}
\item $P_{\alpha,h}$ is a well-defined idempotent operator from $V$ onto $V_h$;
\item the approximation holds:
\begin{multline*}
\|P_{\alpha,h}(y,\uphi{},u,p,\sigma,r)-(y,\uphi{},u,p,\sigma,r)\|_{V}
\\
\leqslant C\inf_{(z_h,\upsi{}_h,v_h,q_h,\tau_h,s_h)\in V_{,h}}\|(y,\uphi{},u,p,\sigma,r)-(z_h,\upsi{}_h,v_h,q_h,\tau_h,s_h)\|_{V_{h}};
\end{multline*}
\item
If $\|P_{\alpha,h}(y,\uphi{},u,p,\sigma,r)-(y,\uphi{},u,p,\sigma,r)\|_{V}\to 0$ as $h\to 0$ for any $(y,\uphi{},u,p,\sigma,r)\in V$, then $\|T_{\alpha,h}-T_\alpha\|_V\to 0$ as $h\to 0$;
\item the operator $T_{\alpha,h}$ is well defined and compact on $V_{h}\subset V$.
\end{enumerate}
\end{lemma}

\paragraph{\bf Example of finite element spaces}
As the Hood-Taylor pair can guarantee {\bf Assumption AIS}, we will consider the group of Lagrangian elements. Denote by $\mathsf{L}^m_h$ the space of continuous piecewise polynomials of m-th degree, and $\mathsf{L}^m_{h0}=\mathsf{L}^m_h\cap H^1_0(\Omega)$, $\mathring{\mathsf{L}}^m_h=\mathsf{L}^m_h\cap L^2_0(\Omega)$. Define
\begin{equation}\label{eq:defvhm}
V_h^m:= \mathsf{L}^{m-1}_h\times (\mathsf{L}^m_{h0})^2\times \mathsf{L}^m_{h0}\times \mathsf{L}^m_h\times \mathring{\mathsf{L}}^m_h\times \mathsf{L}^m_{h0}.
\end{equation}
Then the discretisation \eqref{eq:saddleevpdis} can be implemented with $V_h^m$, and an m-th order accuracy for eigenfunctions and $(2m)$-th order accuracy for eigenvalues can be expected.

\subsection{Implement the  multi-level scheme}\label{ThComplexEvs}
Computing the first several smallest eigenvalues of \eqref{eq:saddleevpdis} is corresponding to computing the first several biggest eigenvalues of $T_{\alpha,h}$ defined by \eqref{eq:defineTh_alpha}, and is fitting for the framework of \textbf{Algorithm 1}. In this subsection, we adopt the algorithm on the eigenvalue problem \eqref{eq:saddleevpdis}.

Note that the eigenvalue problem (\ref{eq:saddleevpdis}) is non-self-adjoint, and special attention has to be paid onto the complex eigenvalues. We begin with the observation below.
\begin{lemma}
Let $a(\cdot,\cdot)$(non-singular) and $b(\cdot,\cdot)$ be two real bilinear forms on real space $V$. If a complex pair $\mu\sim g$ is such that $a(g,w)=\lambda b(g,w)$ for any $w\in V$, then $a(\bar g,w)=\bar \lambda b(\bar g,w)$ for any $w\in V$.
\end{lemma}
\begin{proof}
Denote $\mu=\mu_r+i\mu_i$ with $\mu_r,\mu_i\in\mathbb{R}$, and $g=g_r+ig_i$ with $g_r,g_i\in V$. Then
$$
a(g_r+ig_i,w)=(\mu_r+i\mu_i)b(g_r+ig_i,w)=b(\mu_rg_r-\mu_ig_i,w)+ib(\mu_i g_r+\mu_rg_i,w),
$$
namely
$$
a(g_r,w)=b(\mu_rg_r-\mu_ig_i,w),\quad a(g_i,w)=b(\mu_i g_r+\mu_rg_i,w),
$$
futher, we can obtain
$$
a(g_r-ig_i,w)=b(\mu_rg_r-\mu_ig_i,w)-ib(\mu_i g_r+\mu_rg_i,w)=(\mu_r-i\mu_i)b(g_r-ig_i,w).
$$
The proof is completed.
\end{proof}

In the practical implementation of the algorithm, the finite element spaces on coarse grid will always be enhanced with an approximated eigenfunction and its conjugate vector. This can be realised by enhancing the space with the real and imaginary parts of the vectors respectively.

\section{Numerical experiments}

As in practice, the $n_s>1$ case is of dominant interest \cite{ColtonKress}. In this section, we focus ourselves on this one. The case $n_b<1$ follows similarly. Numerical experiments are conducted on a convex domain (a triangle domain $\Omega_1$, left of Figure \ref{fig:initial_mesh2}) and a non-convex domain (a reshaped L-shaped domain $\Omega_2$, right of Figure \ref{fig:initial_mesh2}). Note that neither domain can be covered by rectangular grids.

We discretize \eqref{eq:relaxbig} with $V_h^m, m=2,3$,  defined as \eqref{eq:defvhm}. Both single- and multi-level algorithms are tested. The initial mesh for $V_h^2$ is $h_0\approx1/8$ (as showed in Figure \ref{fig:initial_mesh2}), while the initial mesh for $V_h^3$ is $h_0\approx1/4$. A series of nested grids $\{\mathcal{T}_{h_i}\}_{i=0}^4$ are constructed by regular bisection refinements with $h_i\approx h_0(1/2)^i$.

\begin{figure}
\centering
\subfigure{\includegraphics[width=0.45\textwidth,height=0.4\textwidth]{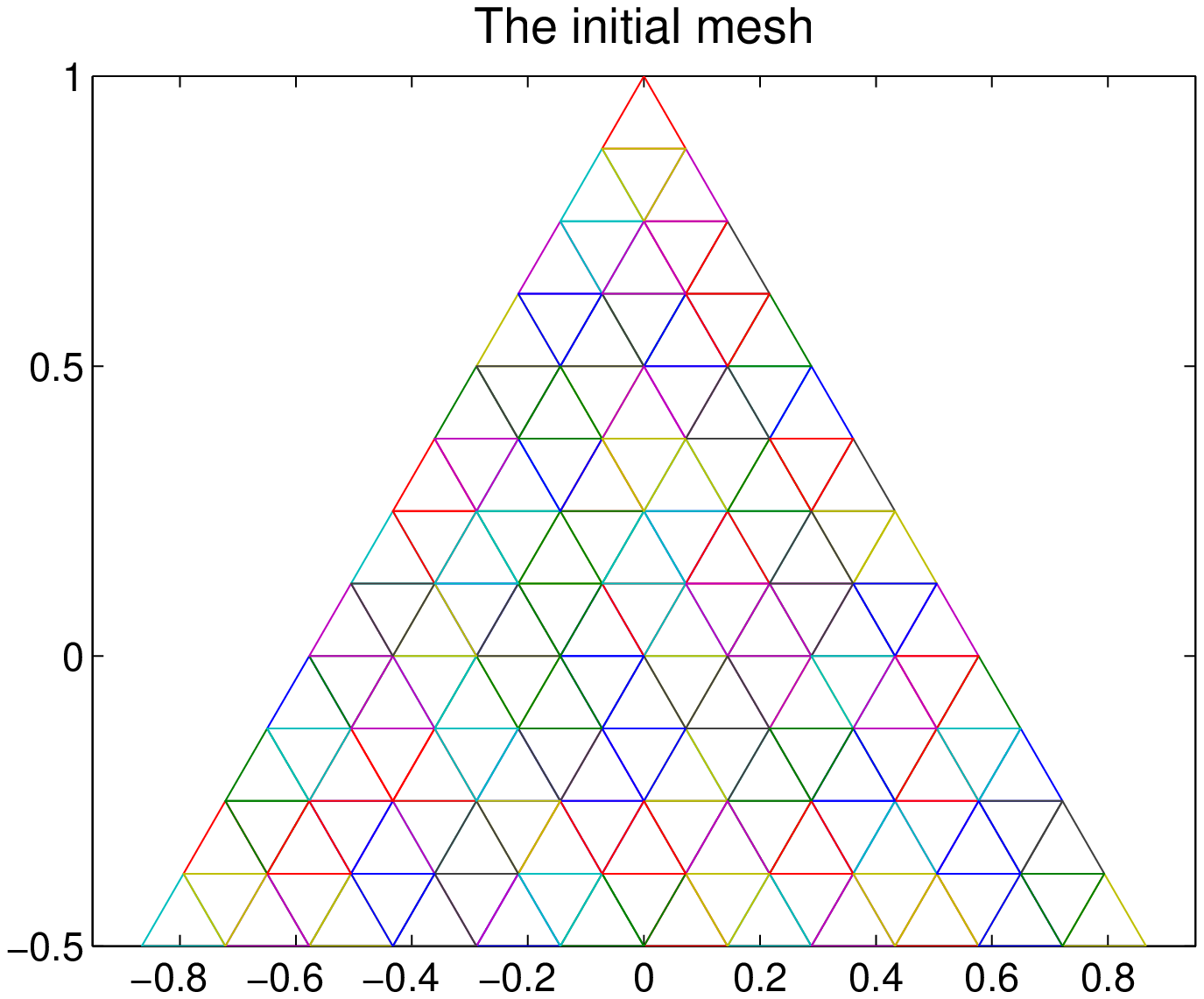}}
\subfigure{\includegraphics[width=0.45\textwidth,height=0.4\textwidth]{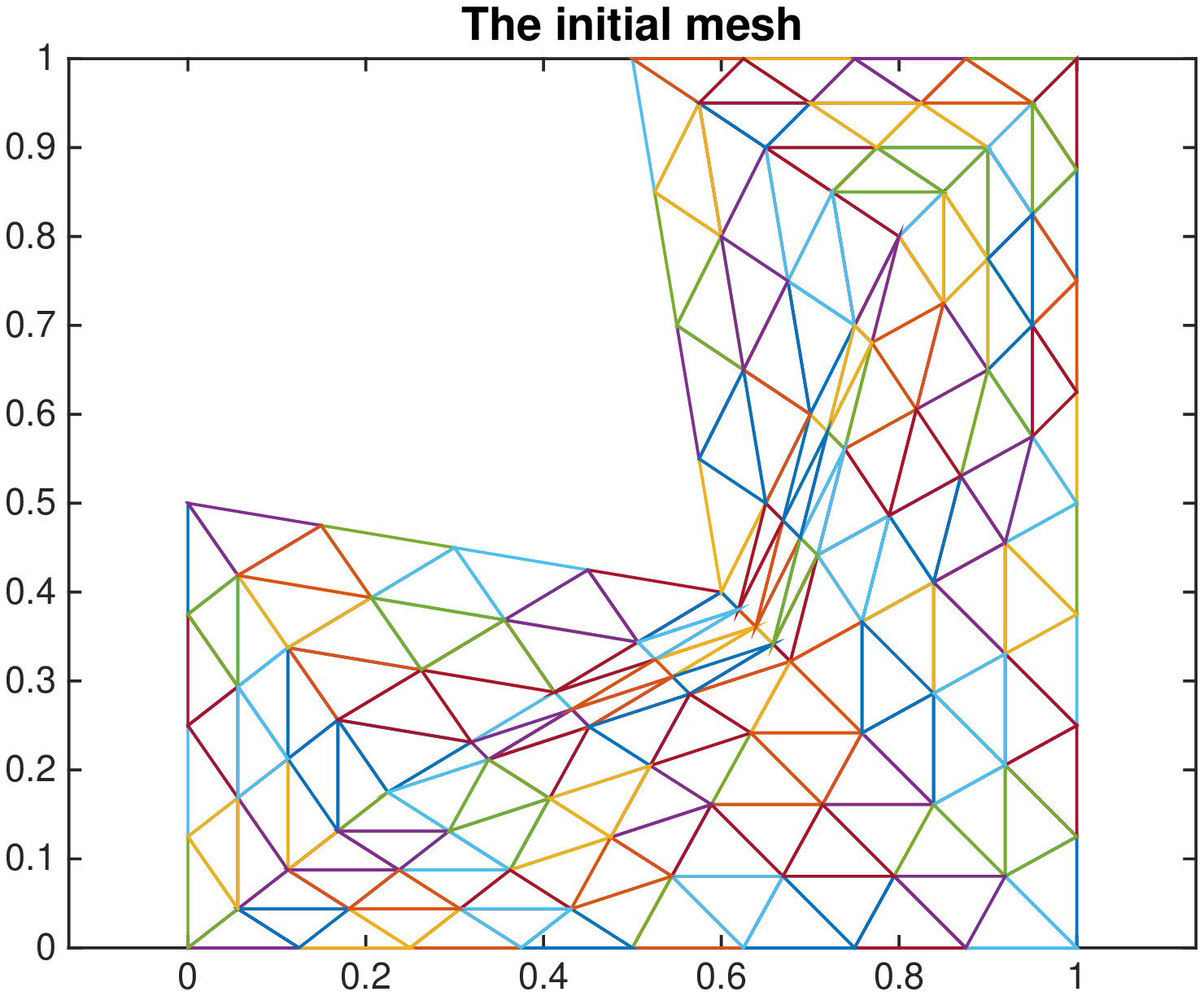}}
\caption{The initial mesh, left: triangle domain ($\Omega_1$), right: the reshaped L-shaped domain ($\Omega_2$).}
\label{fig:initial_mesh2}
\end{figure}

For each series of meshes and every scheme, we obtain the eigenvalue series $\{\lambda_{h_i}\}$ and eigenfunction series $\{(y_{h_i},\uphi{}_{h_i},u_{h_i},p_{h_i},\sigma_{h_i},r_{h_i})\}$. The convergent orders are computed by

\begin{equation}\label{ord_lambda}
\mbox{eigenvalue:}\quad log_2(|\frac{\lambda_{h_4}-\lambda_{h_{i-1}}}{\lambda_{h_4}-\lambda_{h_i}}|),~~~~i=1,2,3.
\end{equation}
\begin{equation}\label{ord_u}
\mbox{component\ $u$ of eigenfunction:}\quad log_2(||\frac{u_{h_4}-u_{h_{i-1}}}{u_{h_4}-u_{h_i}}||_{H^1}),~~~~i=1,2,3, \ \ \mbox{and\ the\ same\ for}\ \uphi.
\end{equation}

\subsection{A summary of numerical experiments}
In the paper, we consider the following examples.

Example 1:  $\Omega_1$ with the index of refraction $n(x)=24$.

Example 2:  $\Omega_1$ with the index of refraction $n(x)=x_1^2+x_2^2+4$.

Example 3:  $\Omega_2$ with the index of refraction $n(x)=16$.

Example 4:  $\Omega_2$ with the index of refraction $n(x)=x_1^2+x_2^2+4$.

We use $V_h^m, m=2,3$, to discretize the problem. For every example, the lowest six eigenvalues are listed in a sequence with the order $``\eqslantless"$. Tables \ref{tab:evs_p2p1_direct} and \ref{tab:evs_p3p2_direct} summarize the results of the single-level $V_h^m$ schemes on the finest meshes. Tables \ref{tab:evs_p2p1_multi} and \ref{tab:evs_p3p2_multi} summarize the results of the multi-level $V_h^m$ schemes on the finest meshes.

From experiments, we can verify the following results.
\begin{enumerate}
\item The performances of the discretization schemes are consistent to the theory.
\item The multi-level algorithms play the same as the corresponding single-level ones. So in Section \ref{NumericalConOrder}, we just give the results for multi-level algorithms.
\item For the convex domain, when the mesh size is small enough, the series of computed {\bf real} eigenvalues tends to decrease monotonously. Namely, a guaranteed upper bound of the eigenvalue can be expected to be computed by the single-level and multi-level algorithms.
\end{enumerate}
 We discuss the convergence behavior of the eigenvalues and eigenfunctions ($u_h$ and $\uphi{}_{h}$) in Section \ref{NumericalConOrder}.

\begin{table}[htp]

\centering
\caption{\label{tab:evs_p2p1_direct} The transmission eigenvalues on the finest mesh with single-level $V_h^2$ scheme ($\lambda=k^2$).}
\begin{tabular}{l|l|l|l}\hline
$\Omega$&$n(x)$&DOFs&The first six eigenvalues\\\hline
\multirow{2}*{$\Omega_1$}&\multirow{2}*{24}&\multirow{2}*{295511}&$\lambda_1=2.1389, \quad\lambda_2=3.4375, \quad\lambda_3=3.4375$, \\
                                            &&&$\lambda_4=5.3173,  \quad\lambda_5=5.3173,  \quad\lambda_6=5.4636$.\\\hline
\multirow{2}*{$\Omega_1$}&\multirow{2}*{$x_1^2+x_2^2+4$}&\multirow{2}*{295511}&$\lambda_1=15.2871-9.2904i,  \quad\lambda_3=25.3979,  \quad\lambda_5=27.9052$, \\
                                            &&&$\lambda_2=15.2871+9.2904i,  \quad\lambda_4=25.3979, \quad\lambda_6=35.5305$.\\\hline
\multirow{2}*{$\Omega_2$}&\multirow{2}*{16}&\multirow{2}*{465671}&$\lambda_1=12.8210,  \quad\lambda_2=14.0055,  \quad\lambda_3=14.5524$, \\
                                            &&&$\lambda_4=15.2335,  \quad\lambda_5=16.7663,  \quad\lambda_6=19.0633$.\\\hline
\multirow{2}*{$\Omega_2$}&\multirow{2}*{$x_1^2+x_2^2+4$}&\multirow{2}*{465671}&$\lambda_1=55.1035,  \quad\lambda_3=45.6601-40.7527i, \quad\lambda_5=56.0094-34.4999i$, \\
&&&$\lambda_2=56.4701, \quad\lambda_4=45.6601+40.7527i,  \quad\lambda_6=56.0094+34.4999i$.\\\hline
\end{tabular}
\end{table}

\begin{table}[htp]

\centering
\caption{\label{tab:evs_p3p2_direct} The transmission eigenvalues on the finest mesh with single-level $V_h^3$ scheme ($\lambda=k^2$).}
\begin{tabular}{l|l|l|l}\hline
$\Omega$&$n(x)$&DOFs&The first six eigenvalues\\\hline
\multirow{2}*{$\Omega_1$}&\multirow{2}*{24}&\multirow{2}*{165175}&$\lambda_1=2.1389,  \quad\lambda_2=3.4375,  \quad\lambda_3=3.4375$, \\
&&&$\lambda_4=5.3172, \quad \lambda_5=5.3172, \quad \lambda_6=5.4636$.\\\hline
\multirow{2}*{$\Omega_1$}&\multirow{2}*{$x_1^2+x_2^2+4$}&\multirow{2}*{165175}&$\lambda_1=15.2871-9.2904i, \quad \lambda_3=25.3979, \quad \lambda_5=27.9052$, \\
&&&$\lambda_2=15.2871+9.2904i,  \quad \lambda_4=25.3979,  \quad \lambda_6=35.5305$.\\\hline
\multirow{2}*{$\Omega_2$}&\multirow{2}*{16}&\multirow{2}*{294151}&$\lambda_1=12.8215,  \quad\lambda_2=14.0055, \quad\lambda_3=14.5521$, \\
&&&$\lambda_4=15.2338, \quad \lambda_5=16.7665, \quad \lambda_6=19.0632$.\\\hline
\multirow{2}*{$\Omega_2$}&\multirow{2}*{$x_1^2+x_2^2+4$}&\multirow{2}*{294151}&$\lambda_1=55.1054, \quad \lambda_3=45.6624-40.7568i,  \quad \lambda_5=56.0097-34.4995i$, \\
&&&$\lambda_2=56.4720, \quad\lambda_4=45.6624+40.7568i, \quad\lambda_6=56.0097+34.4995i$.\\\hline
\end{tabular}
\end{table}

\begin{table}[htp]
\centering
\caption{\label{tab:evs_p2p1_multi} The transmission eigenvalues on the finest mesh with multi-level $V_h^2$ scheme ($\lambda=k^2$).}
\begin{tabular}{l|l|l|l}\hline
$\Omega$&$n(x)$&DOFs&The first six eigenvalues\\\hline
\multirow{1}*{$\Omega_1$}&\multirow{2}*{24}&\multirow{2}*{295511}&$\lambda_1=2.1389,  \quad\lambda_2=3.4375,  \quad\lambda_3=3.4375$, \\
                                            &&&$\lambda_4=5.3173,  \quad\lambda_5=5.3173,  \quad\lambda_6=5.4636$.\\\hline
\multirow{1}*{$\Omega_1$}&\multirow{2}*{$x_1^2+x_2^2+4$}&\multirow{2}*{295511}&$\lambda_1=15.2871-9.2904i, \quad\lambda_3=25.3979, \quad\lambda_5=27.9052$, \\
                                            &&&$\lambda_2=15.2871+9.2904i, \quad\lambda_4=25.3979, \quad\lambda_6=35.5305$.\\\hline
\multirow{1}*{$\Omega_2$}&\multirow{2}*{16}&\multirow{2}*{465671}&$\lambda_1=12.8210,  \quad\lambda_2=14.0055, \quad\lambda_3=14.5524$, \\
                                            &&&$\lambda_4=15.2335,  \quad\lambda_5=16.7663, \quad \lambda_6=19.0633$.\\\hline
\multirow{1}*{$\Omega_2$}&\multirow{2}*{$x_1^2+x_2^2+4$}&\multirow{2}*{465671}&$\lambda_1=55.1035,  \quad\lambda_3=45.6601-40.7527i, \quad\lambda_5=56.0094-34.4999i$, \\
&&&$\lambda_2=56.4701,  \quad\lambda_4=45.6601+40.7527i, \quad\lambda_6=56.0094+34.4999i$.\\\hline
\end{tabular}
\end{table}

\begin{table}[htp]
\centering
\caption{\label{tab:evs_p3p2_multi} The transmission eigenvalues on the finest mesh with multi-level $V_h^3$ scheme ($\lambda=k^2$).}
\begin{tabular}{l|l|l|l}\hline
$\Omega$&$n(x)$&DOFs&The first six eigenvalues\\\hline
\multirow{2}*{$\Omega_1$}&\multirow{2}*{24}&\multirow{2}*{165175}&$\lambda_1=2.1389,\quad \lambda_2=3.4375,\quad \lambda_3=3.4375$, \\
&&&$\lambda_4=5.3172, \quad \lambda_5=5.3172, \quad \lambda_6=5.4636$.\\\hline
\multirow{2}*{$\Omega_1$}&\multirow{2}*{$x_1^2+x_2^2+4$}&\multirow{2}*{165175}&$\lambda_1=15.2871-9.2904i, \quad\lambda_3=25.3979, \quad\lambda_5=27.9052$, \\
&&&$\lambda_2=15.2871+9.2904i, \quad\lambda_4=25.3979,  \quad\lambda_6=35.5305$.\\\hline
\multirow{2}*{$\Omega_2$}&\multirow{2}*{16}&\multirow{2}*{294151}&$\lambda_1=12.8215,  \quad\lambda_2=14.0055,  \quad\lambda_3=14.5521$, \\
&&&$\lambda_4=15.2338, \quad\lambda_5=16.7665,  \quad\lambda_6=19.0632$.\\\hline
\multirow{2}*{$\Omega_2$}&\multirow{2}*{$x_1^2+x_2^2+4$}&\multirow{2}*{294151}&$\lambda_1=55.1054, \quad\lambda_3=45.6624-40.7568i, \quad\lambda_5=56.0097-34.4995i$, \\
&&&$\lambda_2=56.4720, \quad\lambda_4=45.6624+40.7568i, \quad\lambda_6=56.0097+34.4995i$.\\\hline
\end{tabular}
\end{table}

\subsection{Discussion about complex eigenvalues}
For transmission eigenvalue problem (\ref{Problem1}), the non-self-adjointness admits the existence of complex eigenvalues and complex eigenfunctions. The same situation comes across for the discretizations. As real transmission eigenvalues will deserve bigger attention, a question is concerning if a real transmission eigenvalue will be missed in the computation. Experiments show that the sequences of computed complex eigenvalues tend to a complex limit away from the real ax; namely, a {\bf real} transmission eigenvalue can not be approximated by a series of {\bf complex} computed transmission eigenvalues. Therefore, in practical computation, we can adopt such algorithms that focus on the computation of real eigenvalues rather than on all eigenvalues, which may bring convenience.

\subsection{Convergence behavior of the numerical experiments by multi-level algorithm}
\label{NumericalConOrder}
In the following figures, 'p2p1' denotes $V_h^2$ discretization, and 'p3p2' denotes $V_h^3$ discretization. The key feature of the stability of the finite element spaces ({\bf Assumption AIS}) is this way emphasized.  All the results are obtained by the multi-level algorithm.

\subsubsection{Example 1}

Figure \ref{fig:multi_tri24_evs} gives the convergence rates for eigenvalues. The convergence rates for $V_h^2$ are 4 and for $V_h^3$ are 6 which are both optimal. Figure \ref{fig:multi_tri24_efsu} and \ref{fig:multi_tri24_efsphi} give the convergence rates for the eigenfunction components $u_h$ and $\uphi{}_h$ respectively. The convergence rates are 2 for $V_h^2$ and 3 for $V_h^3$ which are consistent with the theoretical expectation. Both single-level and multi-level algorithms give upper bounds for real eigenvalues.


\begin{figure}
\centering
\includegraphics[width=0.75\textwidth,height=0.5\textwidth]{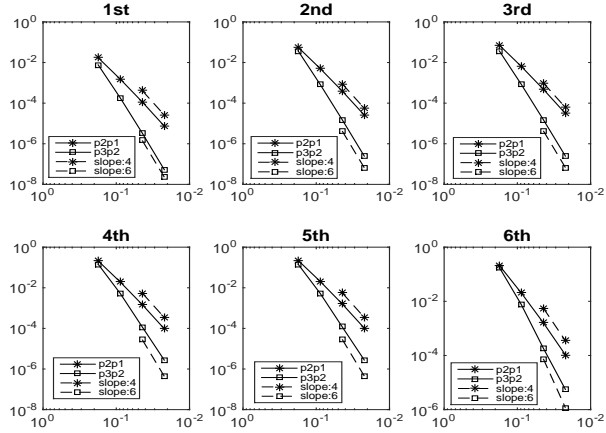}
\caption{The convergence rates for the lowest six eigenvalues for Example 1 by multi-level algorithm, X-axis means the size of mesh and Y-axis means $\vert\lambda_{h_i}-\lambda_{h_4}\vert$.}
\label{fig:multi_tri24_evs}
\end{figure}

\begin{figure}
\centering
\includegraphics[width=0.75\textwidth,height=0.5\textwidth]{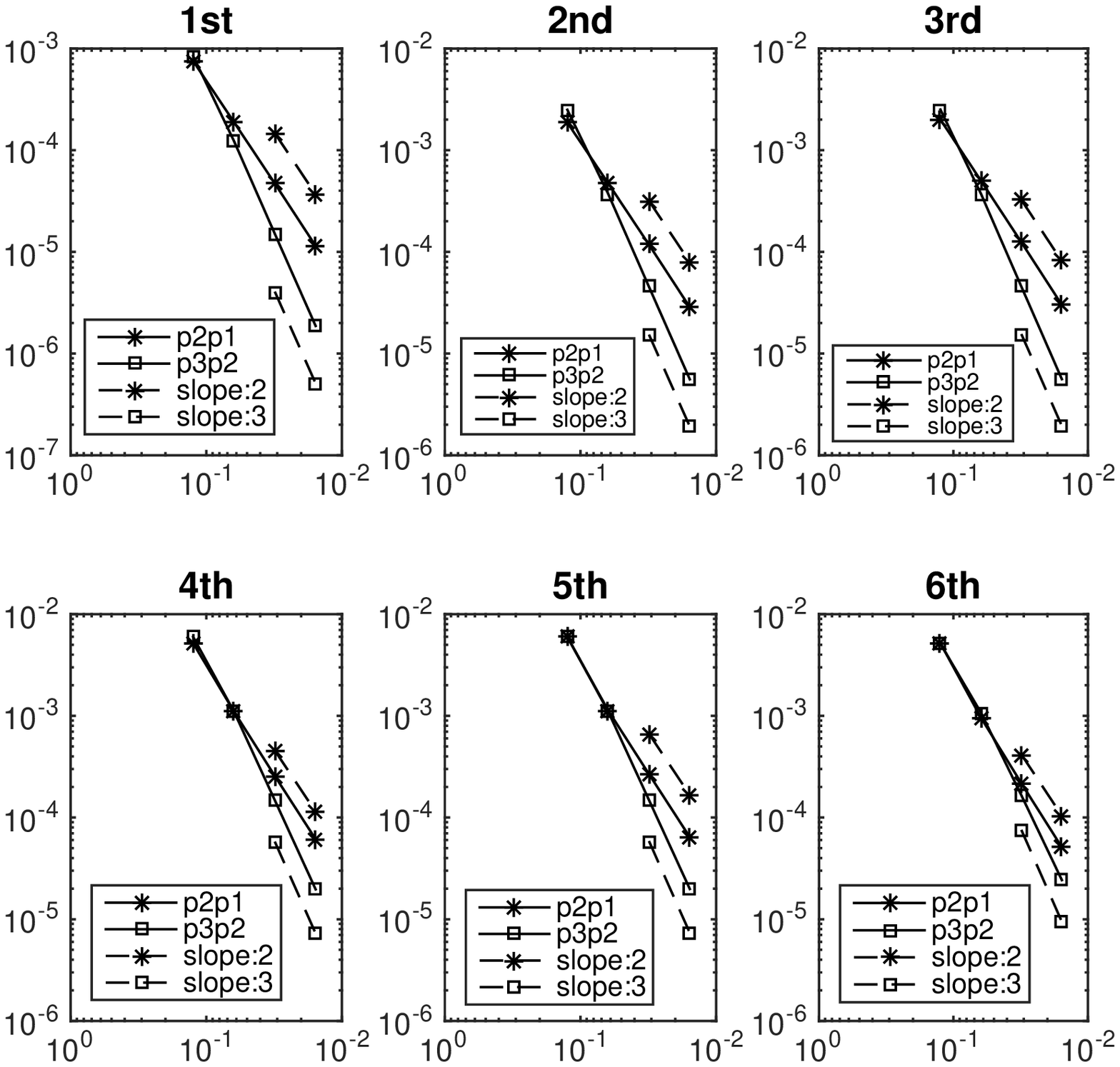}
\caption{The convergence rates for $u_h$ for Example 1 by multi-level algorithm, X-axis means the size of mesh and Y-axis means $||u_{h_i}-u_{h_4}||_{H^1}$.}
\label{fig:multi_tri24_efsu}
\end{figure}

\begin{figure}
\centering
\includegraphics[width=0.75\textwidth,height=0.5\textwidth]{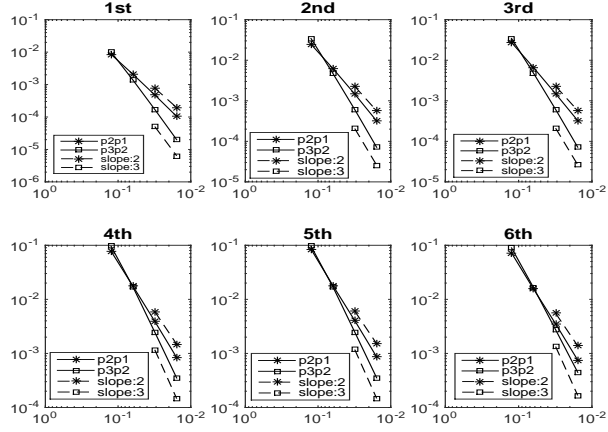}
\caption{The convergence rates for the second component of eigenfunction for Example 1 by multi-level algorithm.}
\label{fig:multi_tri24_efsphi}
\end{figure}

\subsubsection{Example 2}

Figure \ref{fig:multi_triNon2_evs} gives the  convergence rates for eigenvalues which are also optimal.
And the optimal convergence rates for eigenfunction components $u_h$ and $\uphi{}_h$ are also obtained as in Figures \ref{fig:multi_triNon2_efsu} and \ref{fig:multi_triNon2_efsphi}. Again, both algorithms give upper bounds for real eigenvalues.


\begin{figure}
\centering
\includegraphics[width=0.75\textwidth,height=0.5\textwidth]{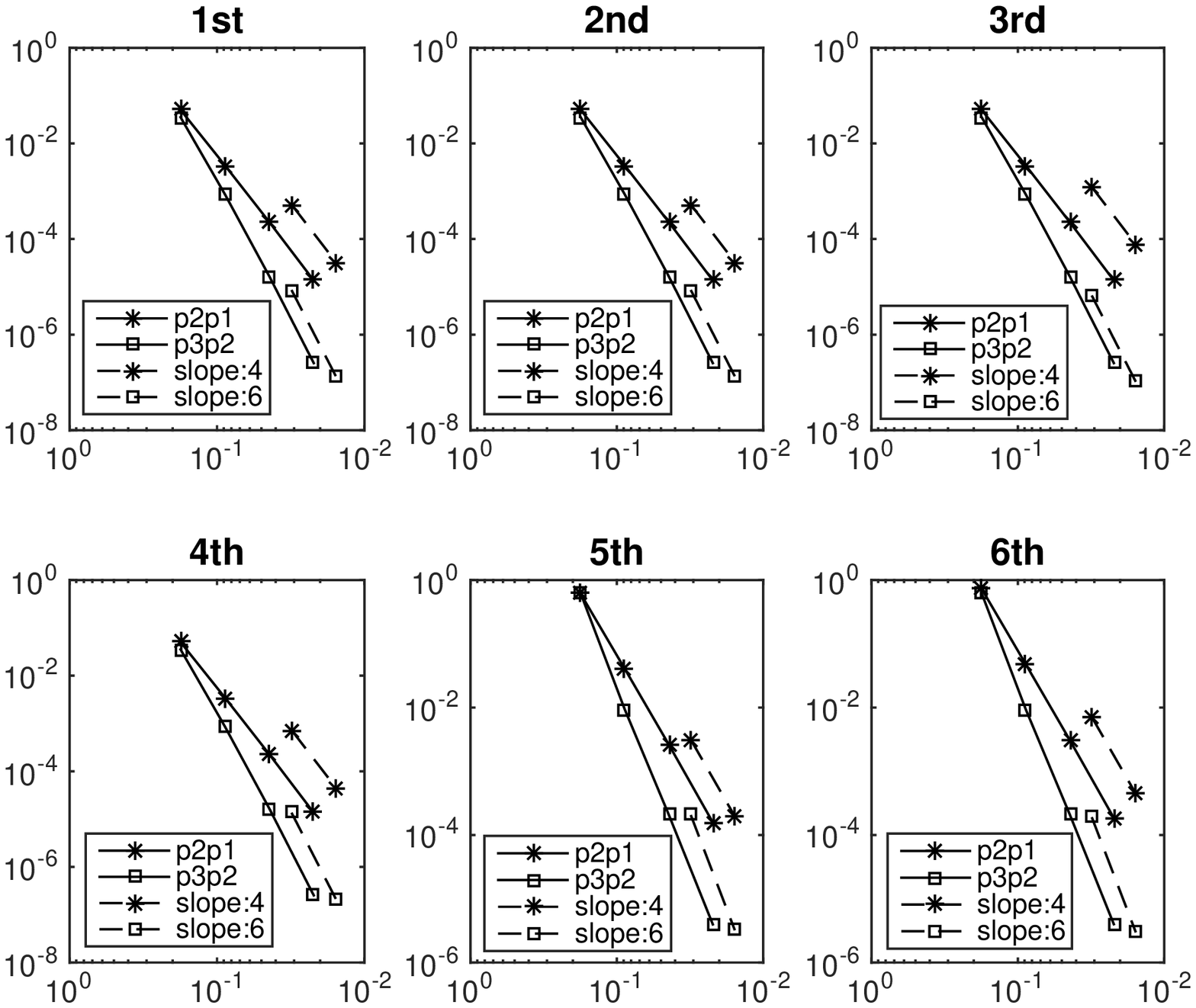}
\caption{The convergence rates for the lowest six eigenvalues for Example 2 by multi-level algorithm, X-axis means the size of mesh and Y-axis means $\vert\lambda_{h_i}-\lambda_{h_4}\vert$.}
\label{fig:multi_triNon2_evs}
\end{figure}

\begin{figure}
\centering
\includegraphics[width=0.75\textwidth,height=0.5\textwidth]{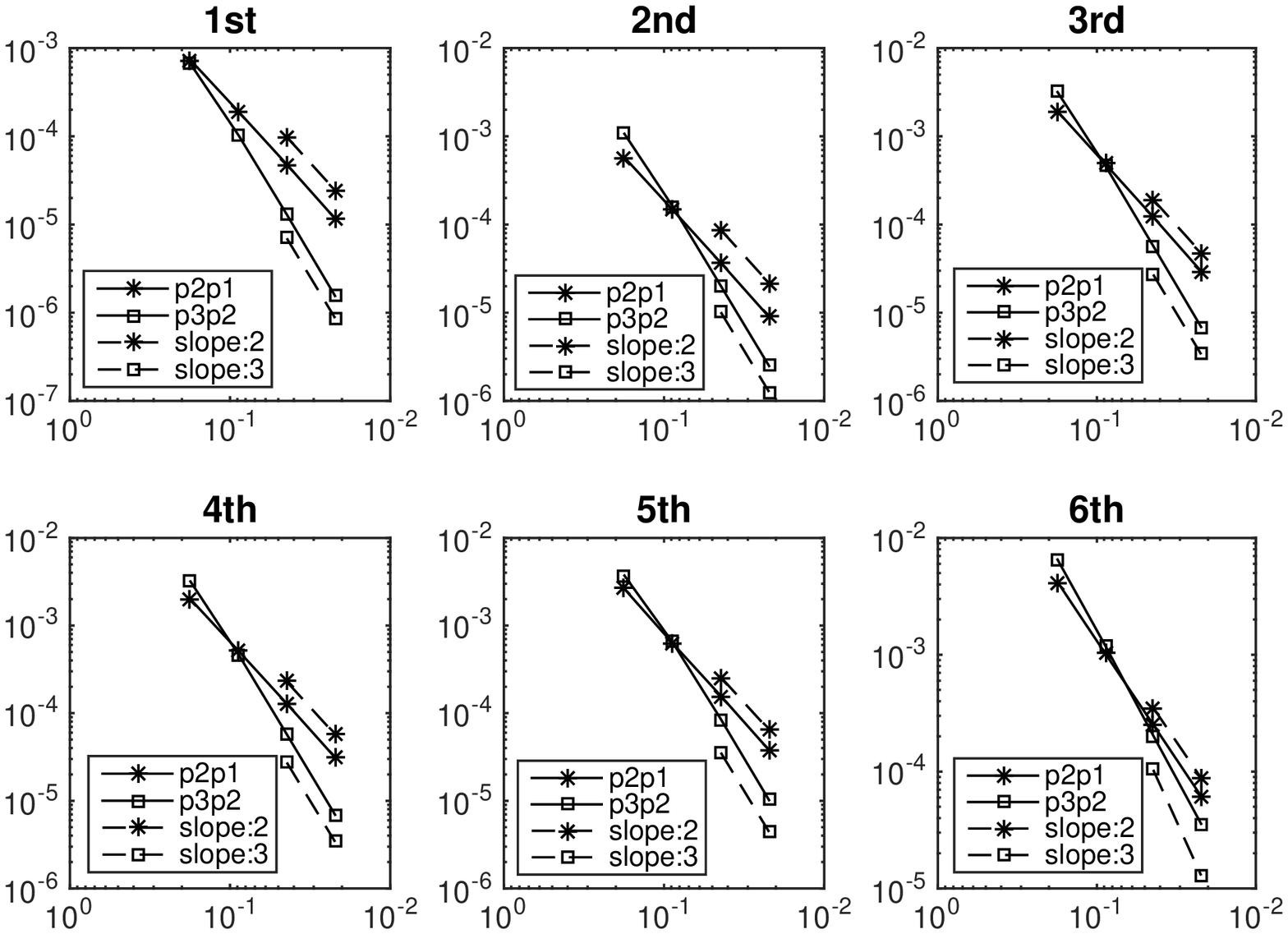}
\caption{The convergence rates for $u_h$ for Example 2 by multi-level algorithm, X-axis means the size of mesh and Y-axis means $||u_{h_i}-u_{h_4}||_{H^1}$.}
\label{fig:multi_triNon2_efsu}
\end{figure}

\begin{figure}
\centering
\includegraphics[width=0.75\textwidth,height=0.5\textwidth]{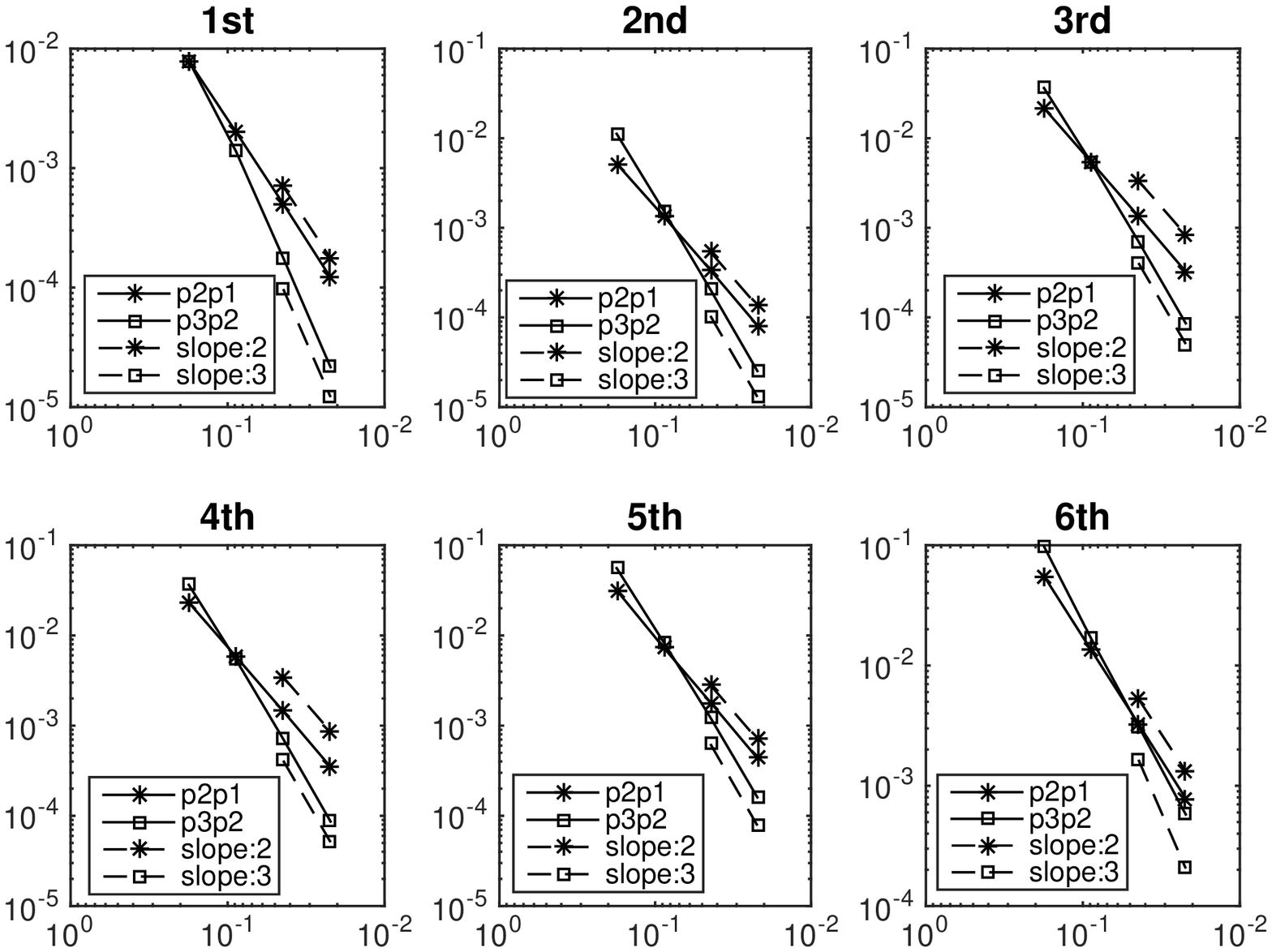}
\caption{The convergence rates for the second component of eigenfunction for Example 2 by multi-level algorithm..}
\label{fig:multi_triNon2_efsphi}
\end{figure}

\subsubsection{Example 3}

Figure \ref{fig:multi_reL16_evs} gives the convergence rates for eigenvalues by multi-level scheme. The rates are not optimal due to the low regularity.
Figures \ref{fig:multi_reL16_efUs}, \ref{fig:multi_reL16_efPhis} give the convergence rates for the eigenfunction components $u_h$ and $\uphi{}_h$, respectively.

\begin{figure}
\centering
\includegraphics[width=0.75\textwidth,height=0.5\textwidth]{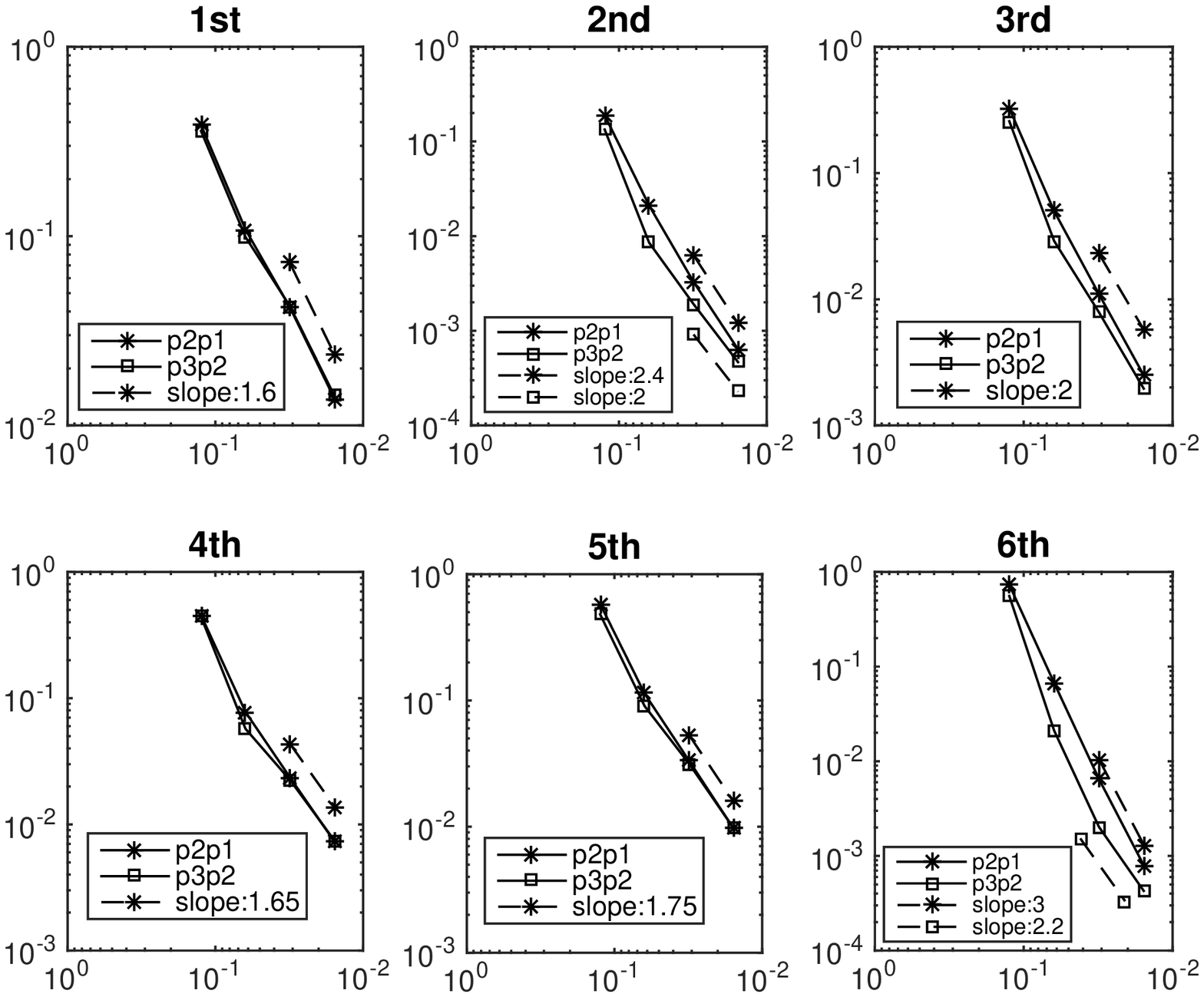}
\caption{The convergence rates for the lowest six eigenvalues for Example 3 by multi-level algorithm, X-axis means the size of mesh and Y-axis means $\vert\lambda_{h_i}-\lambda_{h_4}\vert$.}
\label{fig:multi_reL16_evs}
\end{figure}

\begin{figure}
\centering
\includegraphics[width=0.75\textwidth,height=0.5\textwidth]{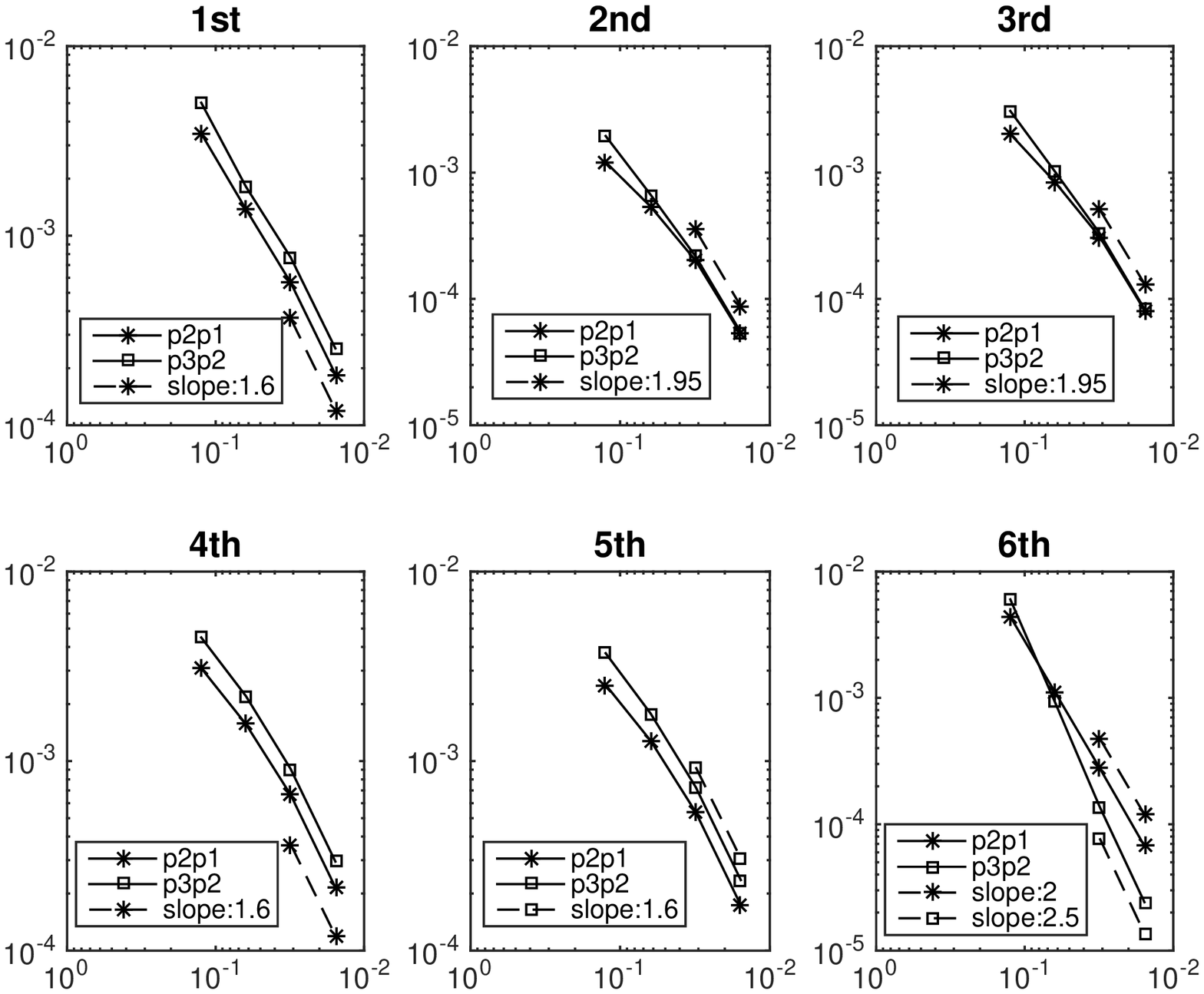}
\caption{The convergence rates for $u_h$ for Example 3 by multi-level algorithm, X-axis means the size of mesh and Y-axis means $||u_{h_i}-u_{h_4}||_{H^1}$.}
\label{fig:multi_reL16_efUs}
\end{figure}

\begin{figure}
\centering
\includegraphics[width=0.75\textwidth,height=0.5\textwidth]{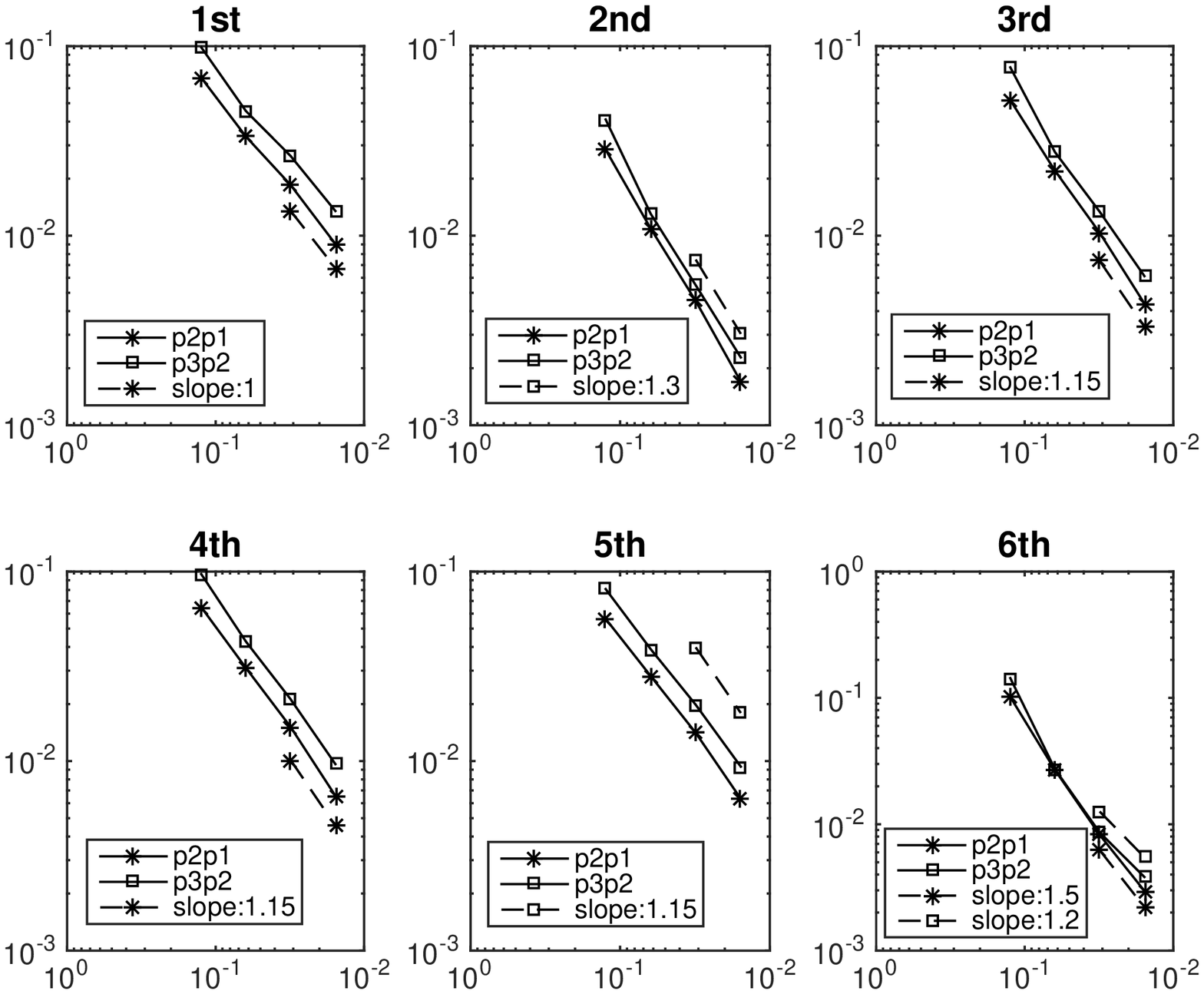}
\caption{The convergence rates for the second component of eigenfunction for Example 3 by multi-level algorithm.}
\label{fig:multi_reL16_efPhis}
\end{figure}

\subsubsection{Example 4}
Figure \ref{fig:multi_reLNon2_evs} gives the convergence rates for eigenvalues by multi-level scheme. The rates are still not optimal.
Figures \ref{fig:multi_reLNon2_efUs}, \ref{fig:multi_reLNon2_efPhis} give the convergence rates for the eigenfunction components $u_h$ and $\uphi{}_h$, respectively.
\begin{figure}
\centering
\includegraphics[width=0.75\textwidth,height=0.5\textwidth]{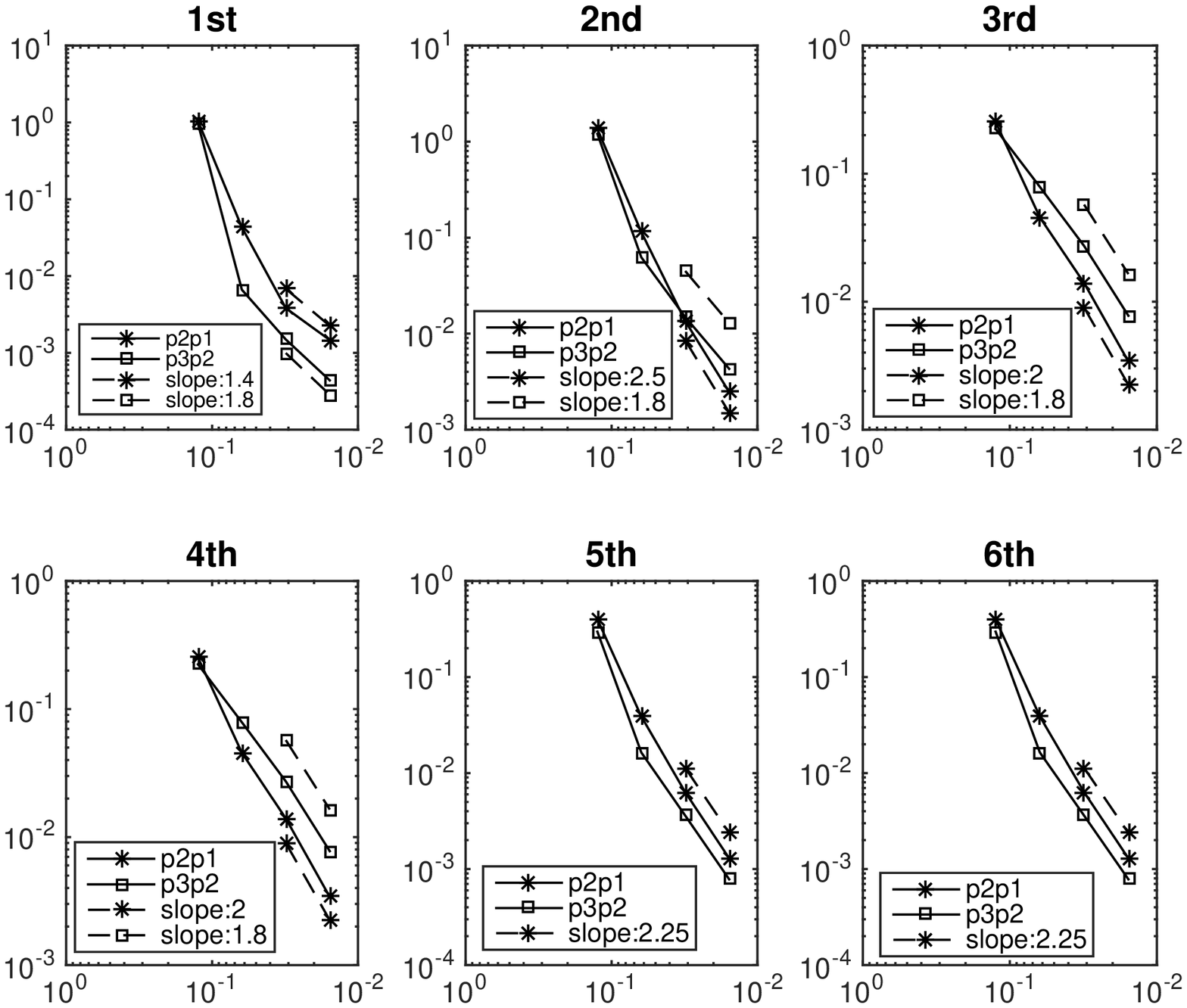}
\caption{The convergence rates for the lowest six eigenvalues for Example 4 by multi-level algorithm, X-axis means the size of mesh and Y-axis means $\vert\lambda_{h_i}-\lambda_{h_4}\vert$.}
\label{fig:multi_reLNon2_evs}
\end{figure}

\begin{figure}
\centering
\includegraphics[width=0.75\textwidth,height=0.5\textwidth]{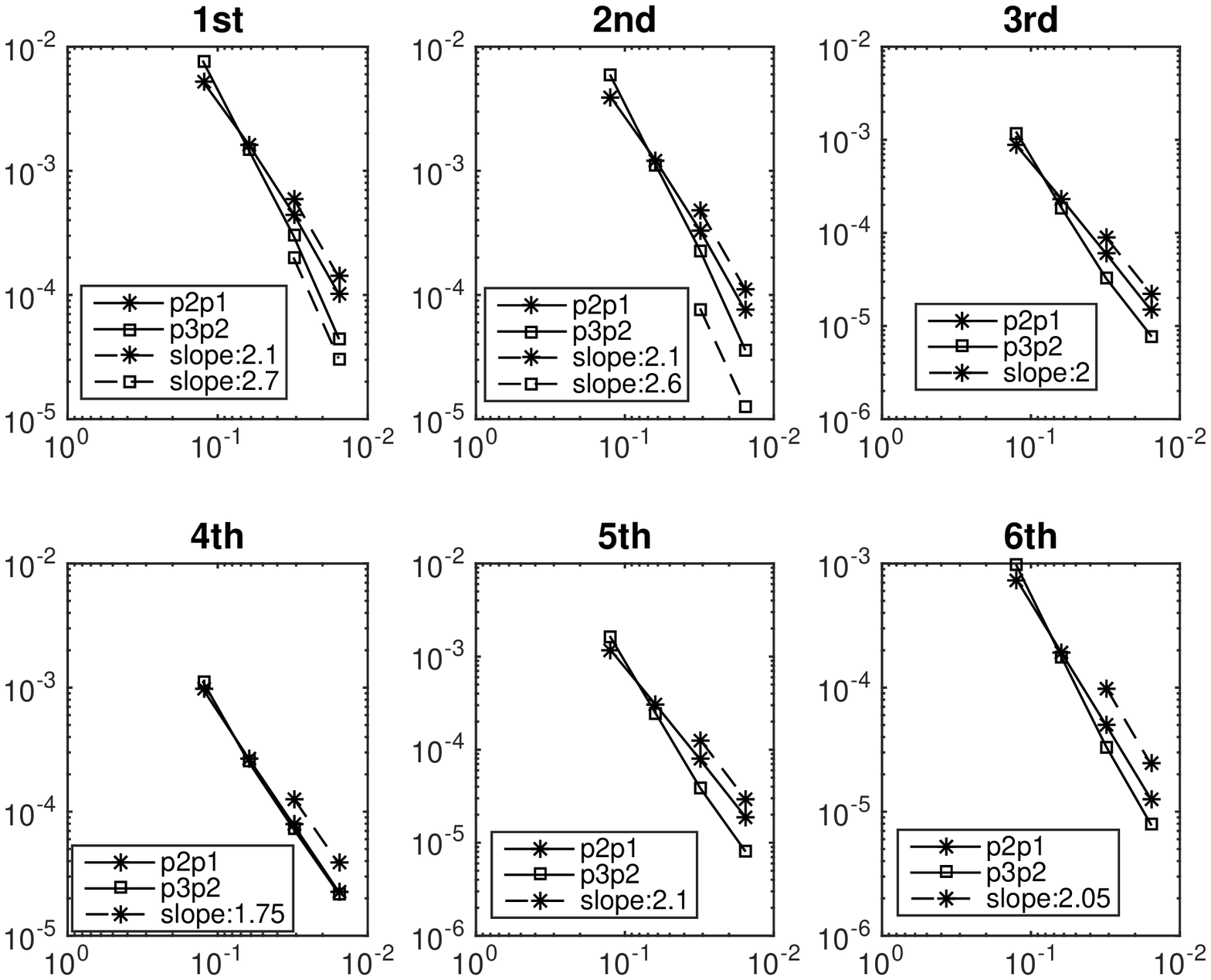}
\caption{The convergence rates for $u_h$ for Example 4 by multi-level algorithm, X-axis means the size of mesh and Y-axis means $||u_{h_i}-u_{h_4}||_{H^1}$.}
\label{fig:multi_reLNon2_efUs}
\end{figure}

\begin{figure}
\centering
\includegraphics[width=0.75\textwidth,height=0.5\textwidth]{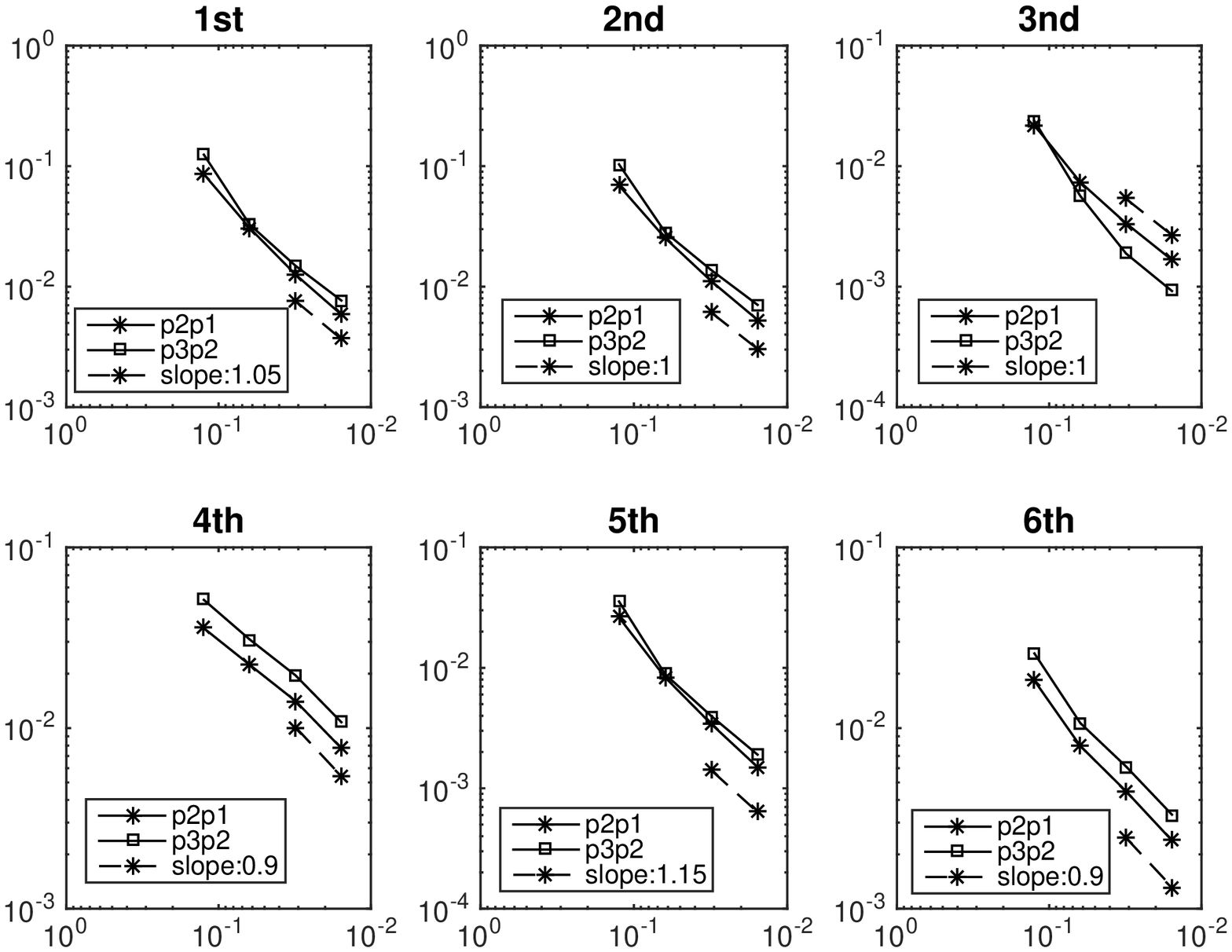}
\caption{The convergence rates for the second component of eigenfunction for Example 4 by multi-level algorithm.}
\label{fig:multi_reLNon2_efPhis}
\end{figure}

\section{Concluding remarks}
In this paper, we discuss the transmission eigenvalue problem discretized by a mixed finite element scheme. The proposed mixed FEM is equivalent to the primal eigenvalue problem. At the continuous level, it doesn't bring in any spurious eigenvalue. The usage of triangular finite elements enables us to deal with arbitrary polygon domain. Particularly, in this paper, we choose conforming Lagrangian FEM for the convenience on constructing a multi-level scheme to improve the efficiency. We remark that, concerning the discretisation alone, nonconforming low-order finite elements may provide different options with interesting properties.

In this paper, we are concerned with the simply-connected two dimensional domains. Principally, similar discussions can be carried on domains with multiply-connected feature and in three dimension. This can be discussed in future. The analogous generalisation of the schemes to other kinds of transmission eigenvalue problems seems natural, such as the elastic transmission eigenvalue problem, which can be discussed in future.   Finally we remark that transmission eigenvalue problems with anisotropic index of refraction would also be of our interest in future.







%
%

\end{document}